\documentclass[12pt]{amsart}
\usepackage[letterpaper,margin=1in]{geometry}
\usepackage[T1]{fontenc}
\usepackage{lmodern,microtype,mathtools,amssymb}
\usepackage{enumitem,needspace}
\usepackage{xcolor,tikz,tikz-cd}
\usepackage[colorlinks=true,linkcolor=blue!45!black,citecolor=blue!45!black,urlcolor=blue!45!black]{hyperref}
\newtheorem{theorem}{Theorem}[section]
\newtheorem{proposition}[theorem]{Proposition}
\newtheorem{lemma}[theorem]{Lemma}
\newtheorem{corollary}[theorem]{Corollary}
\theoremstyle{definition}

\newtheorem{example}[theorem]{Example}
\newtheorem{remark}[theorem]{Remark}
\newcommand{\C}{\mathbb C}
\newcommand{\Z}{\mathbb Z}
\newcommand{\R}{\mathbb R}
\newcommand{\RPP}{\operatorname{RPP}}
\newcommand{\Sym}{\operatorname{Sym}}
\newcommand{\Hom}{\operatorname{Hom}}
\newcommand{\Ext}{\operatorname{Ext}}
\newcommand{\wt}{\operatorname{wt}}
\newcommand{\row}{\operatorname{Row}}
\newcommand{\tr}{\operatorname{tr}}
\newcommand{\Fix}{\operatorname{Fix}}

\newcommand{\LL}{\mathcal L}
\newcommand{\eps}{\varepsilon}
\newcommand{\qbinom}[2]{\genfrac{[}{]}{0pt}{}{#1}{#2}}
\DeclareMathOperator{\pos}{pos}
\setlist[enumerate]{label=\textup{(\roman*)},leftmargin=*,itemsep=3pt}
\title[Quiver bases of Cartan squares]{Quiver bases of Cartan squares of minuscule representations}
\author{David B Rush}
\email{dbr@alum.mit.edu}
\date{}
\hypersetup{pdftitle={Quiver bases of Cartan squares of minuscule representations},pdfauthor={David B Rush}}
\begin{document}
\begin{abstract}
We consider the Cartan square $V^{2\lambda}$ of a minuscule representation $V^\lambda$ of a simply laced complex simple Lie algebra $\mathfrak g$.  We construct for $V^{2\lambda}$ a family of bases, which we call quiver bases, each indexed by the set $\RPP_2(P_\lambda)$ of reverse plane partitions of height two on the minuscule poset $P_{\lambda}$ of $V^{\lambda}$.  

Let $Q$ be a quiver on the Dynkin diagram of $\mathfrak g$, and let $c_Q$ be the corresponding Coxeter element.  The quiver basis $\mathcal B^Q$ is distinguished by the following property: Up to sign, the action of the Tits representative $\dot c_Q$ on $\mathcal B^Q$ lifts the action of $c_Q$, via piecewise-linear toggles, on $\RPP_2(P_\lambda)$.  This proves uniformly that, for any minuscule poset $P$, piecewise-linear Coxeter-motion and rowmotion on $\RPP_2(P)$ exhibit the cyclic sieving phenomenon.  

In type~$A$, the quiver basis for the standard orientation recovers, up to rescaling, the canonical basis, whose compatibility with the long cycle was established by Rhoades.  In other types, however, we show the canonical basis is not compatible with any Coxeter element.  
\end{abstract}
\maketitle

\section{Introduction}

\subsection{Overview}\label{subsec:overview}

The study of reverse plane partitions on minuscule posets encompasses two of the most celebrated cyclic sieving results of recent decades.  

Given a minuscule poset $P$, there exists a simply laced complex simple Lie algebra $\mathfrak{g}$ and a minuscule $\mathfrak{g}$-representation $V$ such that $P$ is the minuscule poset of $V$, meaning that the lattice of order ideals of $P$ is isomorphic to the weight lattice of $V$.  

For a positive integer $m$, let
\[
 \RPP_m(P):=\{\pi:P\to\{0,1,\ldots,m\}:\pi\text{ is order-preserving}\}
\]
be the set of reverse plane partitions (RPPs) of height $m$ on $P$.  To each Coxeter element $c$ of the Weyl group $W$ of $\mathfrak g$ corresponds a composition of piecewise-linear toggles on $\RPP_m(P)$.  Following Okada~\cite{Okada}, we call this map \emph{Coxeter-motion} by $c$ and denote it by $\gamma_c$.  Piecewise-linear rowmotion, denoted by $\row$, also acts on $\RPP_m(P)$ and is conjugate to every $\gamma_c$.\footnote{Rush--Shi~\cite{RS}, Theorem~1.3, proved the conjugacy on order ideals.  Okada~\cite{Okada}, Theorem~15, generalized this to the birational setting, and the piecewise-linear statement follows by tropicalization.}

For the rectangle $P=[k]\times[n-k]$, we may take $\mathfrak g=\mathfrak{sl}_n$ and $V=\bigwedge^k\C^n$, with $W=\mathfrak S_n$.  In 2010, Rhoades~\cite{Rhoades} proved that $\gamma_{(1\,2\,\cdots\,n)}$ exhibits cyclic sieving on $\RPP_m([k]\times[n-k])$ for all $m\geq1$.  The next year, Rush--Shi~\cite{RS} proved uniformly that, for any minuscule poset $P$, every $\gamma_c$ exhibits cyclic sieving on $\RPP_1(P)$.\footnote{These results were stated in different settings; we review their original formulations in Section~\ref{subsec:csp-history}.}

Hopkins~\cite{Hopkins} proposed the simultaneous generalization --- that for every minuscule poset $P$ and positive integer $m$, rowmotion $\row$ (equivalently, Coxeter-motion $\gamma_c$ for any $c$) on $\RPP_m(P)$ exhibits cyclic sieving with respect to the size generating function
\begin{equation}\label{eq:intro-generating-function}
 F_{P,m}(q):=\sum_{\pi\in\RPP_m(P)}q^{|\pi|},
 \qquad |\pi|:=\sum_{p\in P}\pi(p).
\end{equation}

GPT~5.6 Sol~\cite{Sol}, guided by the author, has confirmed Hopkins's conjecture is true.  But a representation-theoretic explanation that demonstrates the occurrence of cyclic sieving to be more than mere combinatorial coincidence remains to be found.  

Let $\lambda$ be the highest weight of $V$.  The irreducible representation $V^{m\lambda}$ of highest weight $m\lambda$ satisfies
\[
 \dim V^{m\lambda}=\#\RPP_m(P).
\]
Let $G$ be the simply connected Lie group with Lie algebra $\mathfrak g$.  For a convincing resolution of Hopkins's conjecture, one seeks to determine, for each Coxeter element $c \in W$, a basis $\mathcal B=\{v_\pi:\pi\in\RPP_m(P)\}$ of $V^{m\lambda}$ and an element $g_c\in G$ that permutes $\mathcal B$ up to scalars as $\gamma_c$ permutes $\RPP_m(P)$:
\[
 g_c\,\C v_\pi=\C v_{\gamma_c(\pi)}
 \qquad\bigl(\pi\in\RPP_m(P)\bigr).
\]
We would call such a $g_c$ a \emph{lift} of $\gamma_c$ up to scalars, and the cyclic sieving polynomial would then come from the character of $V^{m\lambda}$, with the scalar factors taken into account.

To consider the weights of $V$, we have implicitly fixed a Cartan subalgebra $\mathfrak{h} \subset \mathfrak{g}$.  Choose a maximal torus $T \subset G$ with Lie algebra $\mathfrak{h}$.  Then $W\simeq N_G(T)/T$, and each $w\in W$ has a Tits representative $\dot w\in N_G(T)$, defined in Section~\ref{subsec:intro-results}.  Thus, the natural choice for $g_c$ is $\dot c$.  

A natural candidate for $\mathcal{B}$ is Lusztig's canonical basis, which is supported by three historical precedents:
\begin{itemize}
 \item For every $m\geq1$, the Tits representative $\dot w_0$ of the long element of $W$ permutes the canonical basis of $V^{m \lambda}$, up to signs, and lifts the action of the Sch\"utzenberger involution $\xi$ on $\RPP_m(P)$.\footnote{Stembridge~\cite{Stembridge1994}, Theorem~4.1, proved that $\xi$ exhibits the ``$q=-1$ phenomenon'' in 1994 using the standard monomial basis and a representative of $w_0$.  The canonical-basis result follows from relating the Sch\"utzenberger involution on global bases to a quantum Weyl group operator; see Kamnitzer and Tingley~\cite{KT}, Sections~6 and~8, specialized at $q=1$.}
 \item For $\mathfrak{g} = \mathfrak{sl}_n$ and all $1 \leq k \leq n-1$, Rhoades \cite{Rhoades} proved that $\dot c$, for $c=(1\,2\,\cdots\,n)$, permutes the canonical basis of $V^{m \omega_k}$, up to signs, and lifts the action of $\gamma_c$ on $\RPP_m([k] \times [n-k])$ for every $m \geq 1$.  Here $\omega_k$ is the $k$th fundamental weight of $\mathfrak{sl}_n$, and every fundamental weight is minuscule.\footnote{Rhoades stated his result for the dual canonical basis and the long-cycle permutation matrix.  The author's re-proof via restriction of global bases applies equally to the canonical basis; see Rush~\cite{RushGlobal}, Sections~4.2--4.3 and~5.  Replacing the permutation matrix by the Tits representative changes only signs on weight vectors.}
 \item Rush--Shi \cite{RS} proved that $\dot c$ permutes the canonical basis of $V^{\lambda}$, up to signs, and lifts the action of $\gamma_c$ on $\RPP_1(P)$.\footnote{For a minuscule representation, the canonical basis is the standard weight basis, with one vector for each weight and simple lowering operators carrying one basis vector to another; see Geck~\cite{Geck}, Proposition~2.9.  Tits representatives act on this basis by the corresponding Weyl group elements, up to signs.  Rush--Shi~\cite{RS}, Theorem~1.4, identify this permutation of weights with Coxeter-motion.}
\end{itemize}

Outside type~$A$, however, the canonical basis fails at height two: 
No Tits representative $\dot c$ of a Coxeter element permutes the canonical basis of $V^{2\lambda}$ even up to nonzero scalar multiples (Theorem~\ref{thm:canonical-obstruction}).  Rescaling its vectors therefore cannot produce a compatible basis.  

In this article, we completely resolve the $m=2$ case by allowing the basis $\mathcal{B}$ to vary with the Coxeter element $c$.  Each orientation $Q$ of the Dynkin diagram of $\mathfrak{g}$ determines a Coxeter element $c_Q$, and we construct a \emph{quiver basis} $\mathcal B^Q$ of the Cartan square $V^{2\lambda}$, indexed by $\RPP_2(P)$, on which $\dot c_Q$ lifts $\gamma_{c_Q}$ up to signs.  This yields a uniform, representation-theoretic proof that $\gamma_c$ exhibits cyclic sieving on $\RPP_2(P)$ for every minuscule poset $P$ and Coxeter element $c$.  

With his 1994 ``$q = -1$'' theorem, Stembridge initiated the program of finding, for Cartan powers $V^{m \lambda}$ of minuscule representations, bases compatible with natural cyclic actions on reverse plane partitions of minuscule posets.  After more than thirty years, this work is the first advance for Coxeter-motion beyond rectangles and height one --- and is uniform in both the minuscule poset and the Coxeter element.  A compatible-basis construction for arbitrary height is still open.  

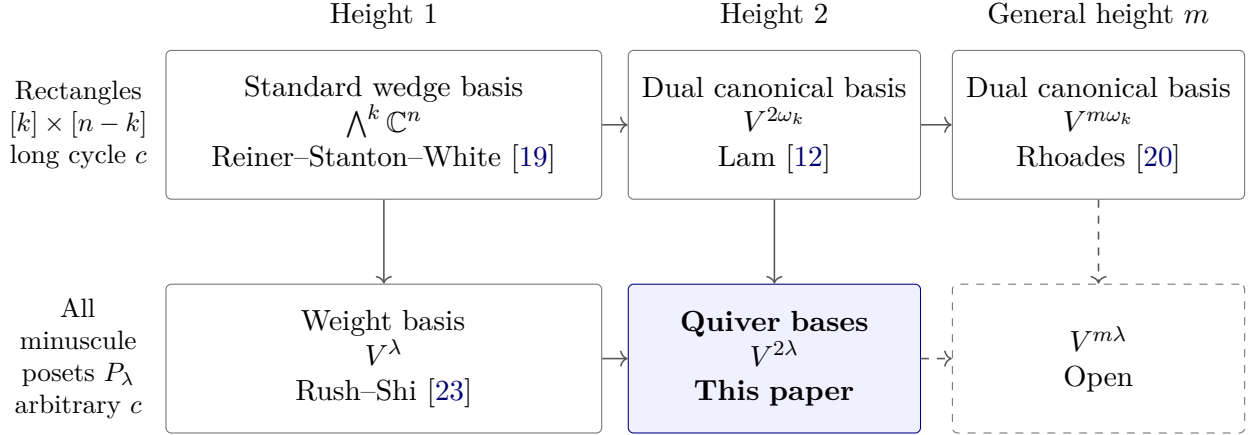
\begin{figure}[htbp]
\centering
\begin{tikzpicture}[x=1pt,y=1pt,
every node/.style={font=\small},
case/.style={draw=black!55,rounded corners=2pt,minimum height=56pt,align=center,inner sep=3pt},
generalization/.style={->,draw=black!65,semithick}]
\node at (142,41.0) {Height $1$};
\node at (289,41.0) {Height $2$};
\node at (411,41.0) {General height $m$};
\node[align=center,text width=52pt,font=\footnotesize] at (27,0) {Rectangles\\$[k]\times[n-k]$\\long cycle $c$};
\node[align=center,text width=52pt,font=\footnotesize] at (27,-88) {All minuscule\\posets $P_\lambda$\\arbitrary $c$};
\node[case,text width=158pt] (prototype) at (142,0) {Standard wedge basis\\$\bigwedge^k\C^n$\\\mbox{Reiner--Stanton--White~\cite{RSW}}};
\node[case,text width=104pt] (typeAtwo) at (289,0) {Dual canonical basis\\$V^{2\omega_k}$\\Lam~\cite{LamDimers}};
\node[case,text width=104pt] (typeAall) at (411,0) {Dual canonical basis\\$V^{m\omega_k}$\\Rhoades~\cite{Rhoades}};
\node[case,text width=158pt] (minusculeone) at (142,-88) {Weight basis\\$V^\lambda$\\Rush--Shi~\cite{RS}};
\node[case,text width=104pt,draw=blue!50!black,fill=blue!6] (minuscletwo) at (289,-88) {\textbf{Quiver bases}\\$V^{2\lambda}$\\\textbf{This paper}};
\node[case,text width=104pt,dashed] (minusculeall) at (411,-88) {$V^{m\lambda}$\\Open};
\draw[generalization] (prototype) -- (typeAtwo);
\draw[generalization] (typeAtwo) -- (typeAall);
\draw[generalization] (typeAtwo) -- (minuscletwo);
\draw[generalization] (prototype) -- (minusculeone);
\draw[generalization] (minusculeone) -- (minuscletwo);
\draw[generalization,dashed] (minuscletwo) -- (minusculeall);
\draw[generalization,dashed] (typeAall) -- (minusculeall);
\end{tikzpicture}
\caption{Compatible bases for Coxeter-motion on $\RPP_m(P)$.  Horizontal arrows extend the range of heights; vertical arrows extend the range of minuscule posets and Coxeter elements.  The highlighted box is the height-two result of this paper; the dashed box is the open problem at arbitrary height.  Following Rhoades's work on the type~$A$ dual canonical basis, Lam~\cite{Lam} gave a combinatorial realization at height two using noncrossing matchings.  Our quiver bases generalize his construction, as we show in Section~\ref{sec:comparison}.}
\label{fig:motivation}
\end{figure}

\subsection{Cyclic sieving and representation theory}\label{subsec:csp-history}

Figure~\ref{fig:motivation} delineates the settings for which we consider Coxeter-compatible bases, starting from height-one RPPs on the rectangle $[k] \times [n-k]$.  We describe this prototypical example and its complementary generalizations due to Rhoades~\cite{Rhoades} and Rush--Shi~\cite{RS} as they were originally formulated, and we explain how equivalent formulations fit into the RPP framework.  

We begin with the upper-left box of Figure~\ref{fig:motivation}.  Fix $P=[k]\times[n-k]$, with $1\leq k\leq n-1$.  Let $S$ be the set of $k$-element subsets of $[n]$.  A subset $A=\{a_1<\cdots<a_k\}\in S$ determines a partition $(a_k-k,\ldots,a_1-1)$, which determines an order ideal $I_A \subset P$, and hence a height-one RPP $\pi_A:=\mathbf1_{P\setminus I_A}$.  Under this correspondence, the long cycle $c=(1\,2\,\cdots\,n)$ on subsets becomes $\gamma_c$ on RPPs.  Furthermore, the Gaussian polynomial for $S$ coincides with the size generating polynomial for $\RPP_1(P)$:
\[
 f_S(q):=\qbinom{n}{k}_q
 =\sum_{1\leq a_1<\cdots<a_k\leq n}q^{\sum_{r=1}^k(a_r-r)}
 =\sum_{I_A \subset P} q^{|I_A|}
 =\sum_{\pi_A \in \RPP_1(P)} q^{|P|-|\pi_A|}
 =F_{P,1}(q),
\]
where the last equality follows from the self-duality of the rectangle.\footnote{Order ideals correspond directly, with their usual size grading, to order-reversing $\{0,1\}$-valued maps, but the order-preserving convention is more common in the literature.  See Appendix~\ref{app:classical-conventions} for more details on these conventions.}  At $q=1$, the polynomial counts all $k$-subsets, and, at $n$th roots of unity, Reiner--Stanton--White~\cite{RSW} showed
\begin{equation}\label{eq:csp-definition}
 \#\Fix(c^d)=f_S(\zeta_n^d)
 \qquad(d\in\Z),\qquad \zeta_n=\exp(2\pi i/n).
\end{equation}
This relationship defines the \emph{cyclic sieving phenomenon}: Root-of-unity evaluations of a generating function count fixed points of the corresponding powers of a cyclic action.

Consider the $GL_n(\C)$-representation $\bigwedge^k\C^n$, with standard basis $\mathcal B=\{e_A:A\in S\}$, where $e_A:=e_{a_1}\wedge\cdots\wedge e_{a_k}$.  Let $g \in GL_n(\C)$ be the permutation matrix given by $ge_j=e_{j+1}$ for $j<n$ and $ge_n=e_1$.  Then
\[
 ge_A=(-1)^{(k-1)\mathbf1_{n\in A}}e_{cA}.
\]
Thus, up to sign, the action of $g$ on $\mathcal{B}$ lifts the action of $c$ on $S$, and, therefore, the action of $\gamma_c$ on $\RPP_1(P)$.

The character of $\bigwedge^k \C^n$ is the elementary symmetric polynomial $e_k$ in $n$ variables, with principal specialization $e_k(1,q,\ldots,q^{n-1})=q^{\binom{k}{2}}f_S(q)$.  The eigenvalues of $g$ are $1,\zeta_n,\ldots,\zeta_n^{n-1}$, and were $g$ to permute $\mathcal B$ exactly as $c$ permutes $S$, traces of its powers would count fixed points directly.  Accounting for the wedge signs\footnote{For the details, see Appendix~\ref{app:classical-conventions}.} gives
\[
 \#\Fix(c^d)=\zeta_n^{-d\binom{k}{2}}
 \tr\!\left(g^d\mid\bigwedge^k\C^n\right)=f_S(\zeta_n^d).
\]
This demonstrates the representation-theoretic paradigm for cyclic sieving --- the crux is constructing a basis compatible with the cyclic action, after which obtaining the cyclic sieving polynomial from the representation's character, taking scalar factors into account, is straightforward.

We move to the upper-right box of Figure~\ref{fig:motivation}.  For $m\geq1$, the set $S_m$ of semistandard Young tableaux of shape $(m^k)$ with entries in $[n]$ generalizes $S$ (identify a one-column tableau with its entries to see $S_1 \simeq S$).  For $T\in S_m$, regard each column $T^{(s)}$ as a subset in $S$, and set $I_s:=I_{T^{(s)}}$ for $1 \leq s \leq m$.  Since the entries of $T$ weakly increase across rows, $I_1\subseteq\cdots\subseteq I_m$, and this chain determines an RPP
\[
 \pi_T:=\sum_{s=1}^m\mathbf1_{P\setminus I_s}\in\RPP_m(P).
\]
The map $S_m \rightarrow \RPP_m(P)$ given by $T\mapsto\pi_T$ has an inverse, which recovers the order ideals via $I_s=\{p:\pi(p)<s\}$ and then the tableau $T$ by writing the subset corresponding to $I_s$ in column $s$.  Under this bijection, Sch\"utzenberger's jeu-de-taquin promotion\footnote{We follow Rhoades's convention, in which promotion removes the entries equal to $n$, slides the holes northwest, increases the remaining entries by one, and fills the holes with $1$'s.  At $m=1$, this is the action of $c$ on $S$.} $j$ on $S_m$ becomes $\gamma_c$ on $\RPP_m(P)$; cf. Hopkins~\cite{HopkinsPromotion}, Appendix~A.

The irreducible polynomial $\operatorname{GL}_n(\C)$-representation of highest weight $(m^k)$, which we denote by $S^{(m^k)}\C^n$, generalizes $\bigwedge^k \C^n$.  Rhoades~\cite{Rhoades} proved that $g$ acts on its dual canonical basis $\{F_T:T\in S_m\}$ by promotion (up to signs).\footnote{The restriction of $S^{(m^k)} \C^n$ to $\operatorname{SL}_n(\C)$ is $V^{m\omega_k}$.  The lifting property holds with $\dot c \in \operatorname{SL}_n(\C)$ in place of $g$; see Section~\ref{subsec:intro-results}.}  Explicitly, if $T \in S_m$ has content $(T_1,\ldots,T_n)$, then
\[
 gF_T=(-1)^{(k-1)T_n}F_{jT}.
\]
The character of $S^{(m^k)} \C^n$ is the Schur polynomial $s_{(m^k)}$ in $n$ variables, and $j$ has order $n$ and exhibits cyclic sieving on $S_m$ with respect to the polynomial
\begin{equation}\label{eq:rhoades-polynomial}
	f_{S_m}(q):=q^{-m\binom{k}{2}}s_{(m^k)}(1,q,\ldots,q^{n-1})
	=F_{P,m}(q),
\end{equation}
where the last equality again follows from the self-duality of the rectangle.\footnote{See Appendix~\ref{app:classical-conventions} for the details.}  

Subsequently, Shen and Weng~\cite{ShenWeng} reproved Rhoades's cyclic sieving result by showing a corresponding lifting result for the theta basis, and the author~\cite{RushGlobal} obtained a concise proof of Rhoades's lifting result by restriction of global bases.\footnote{The result is stated in~\cite{RushGlobal} for Kashiwara's upper global basis.  One advantage of the author's approach is that the same argument applies equally to the lower and upper global bases, which coincide with Lusztig's canonical and dual canonical bases, respectively; see Grojnowski and Lusztig~\cite{GrojnowskiLusztig}.  Throughout this article, we refer to the $q=1$ specialization of these bases.}

We turn to the lower-left box of Figure~\ref{fig:motivation}.  For any minuscule poset $P$, Rush--Shi~\cite{RS} work with its lattice of order ideals $J(P)$, which maps bijectively to $\RPP_1(P)$ by $I\mapsto\mathbf1_{P\setminus I}$.  To toggle an order ideal $I \in J(P)$ at an element $p \in P$ means to return the symmetric difference $I \triangle \{ p\}$ if $I \triangle \{ p\} \in J(P)$ and $I$ otherwise.\footnote{Under the bijection $J(P) \rightarrow \RPP_1(P)$, toggling at $p$ becomes piecewise-linear toggling at $p$; thus, piecewise-linear toggling on $\RPP_m(P)$ generalizes toggling on $J(P)$.}  Both rowmotion and Coxeter-motion can be expressed as compositions of toggles, with the toggle at each $p\in P$ applied exactly once.  For rowmotion, the toggles are applied along a linear extension of $P$ in reverse order.  For a Coxeter element $c=s_{i_r}\cdots s_{i_1}$, Coxeter-motion applies the toggles at elements labeled $i_1$, then those at elements labeled $i_2$, and so on.\footnote{The elements of $P$ are naturally labeled by the vertices of the Dynkin diagram of $\mathfrak{g}$; following Stembridge~\cite{StembridgeMin}, we refer to this as the heap labeling and to the labeled poset $P$ as a minuscule heap.}  Rush--Shi prove that these actions are conjugate and exhibit cyclic sieving at $h$th roots of unity, where $h$ is the Coxeter number of $\mathfrak{g}$, with respect to the polynomial
\[
 J(P;q):=\sum_{I\in J(P)}q^{|I|}=\sum_{\pi \in \RPP_1(P)} q^{|P| - |\pi|} = F_{P,1}(q).
\]

Their proof, although formulated in the language of fully commutative Weyl group elements, is in keeping with the representation-theoretic paradigm.  Let $\wt\colon J(P)\to W\lambda$ be the bijection between the lattice of order ideals of $P$ and the weight lattice of the corresponding minuscule representation $V=V^\lambda$.  In this language, Theorem~1.4 of Rush--Shi~\cite{RS} asserts that $\wt$ intertwines Coxeter-motion and the action of $c$:
\[
 c\wt(I)=\wt(\gamma_c I).
\]
Since the weight spaces of $V^\lambda$ are one-dimensional, any representative of $c$ lifts Coxeter-motion on the weight spaces; on the standard weight basis (which coincides with the canonical and dual canonical bases), the Tits representative $\dot c$ lifts $\gamma_c$ up to signs.  

In this article, we focus on the center box in the bottom row of Figure~\ref{fig:motivation}, which sits at the intersection of the Rhoades~\cite{Rhoades} and Rush--Shi~\cite{RS} stories.  As in the tableau correspondence, an RPP determines nested threshold ideals; as in the height-one case, those ideals determine weights of $V^\lambda$.  Thus, at height two, each RPP corresponds to a pair of weights.  In the next subsection, we construct a vector of $V^{2\lambda}$ from each pair, in accordance with a quiver on the Dynkin diagram of $\mathfrak{g}$ --- and explain how these vectors form a quiver basis compatible with Coxeter-motion.

\subsection{Quiver bases}\label{subsec:intro-construction}

We construct the quiver bases inside the symmetric square of a minuscule representation.  Fix a dominant minuscule weight $\lambda$ of a simply laced complex simple Lie algebra $\mathfrak g$.  Let $V^{\lambda}$ be the irreducible $\mathfrak{g}$-representation with highest weight $\lambda$; let $P=P_\lambda$ be the minuscule poset of $V^{\lambda}$, and let $\{u_\mu:\mu\in W\lambda\}$ be its standard weight basis.  

Note that $V^{2\lambda}$ is the $\mathfrak{g}$-submodule of $\Sym^2V^\lambda$ generated by $u_\lambda^2$ and contains every square $u_\mu^2$.  A reverse plane partition $\pi\in\RPP_2(P)$ determines two nested threshold ideals
\[
 A:=\{p:\pi(p)<1\}\subseteq B:=\{p:\pi(p)<2\},
\]
and hence two weights $\mu:=\wt(A)$ and $\nu:=\wt(B)$ such that $\mu \geq \nu$ in the root order.\footnote{The root order on the weight lattice of $\mathfrak g$ is the partial order generated by the covering relations $\mu\gtrdot\nu$ for all $\mu, \nu$ such that $\mu-\nu$ is a simple root.}  Write $\mu-\nu=\sum_i d_i\alpha_i$ in the basis of simple roots.  

Fix an orientation $Q$ of the Dynkin diagram of $\mathfrak g$.  A representation of $Q$ with dimension vector $d = (d_i)$ assigns $\C^{d_i}$ to vertex $i$ and a linear map to each arrow.  These representations form the affine space
\[
 \operatorname{Rep}_Q(d):=\prod_{i\to j}\operatorname{Hom}(\C^{d_i},\C^{d_j}),
\]
on which $\prod_i\operatorname{GL}_{d_i}(\C)$ acts by changes of bases at the vertices.  By Gabriel's theorem~\cite{Gabriel}, this action has finitely many orbits.  Since $\operatorname{Rep}_Q(d)$ is irreducible, exactly one such orbit is Zariski-open.\footnote{By irreducibility, that there are finitely many orbits implies one orbit is dense; since orbits are locally closed, the dense orbit is open.  Two open orbits would intersect, so there can only be one.}  We show that the indecomposable summands of a \textit{generic} representation --- that is, a representation in the unique open orbit --- have as their dimension vectors distinct, pairwise-orthogonal positive roots.  Write $R_Q(\mu-\nu)$ for this set of roots, whose sum is $\mu-\nu$.

Write $R:=R_Q(\mu-\nu)$.  For every subset $S\subseteq R$, the expression
\[
 \mu_S:=\mu-\sum_{\beta\in S}\beta
\]
is a weight of $V^\lambda$.  These weights form the vertices of a cube $C$, with $\mu_\varnothing=\mu$ and $\mu_R=\nu$ as opposite vertices.

For each $\beta\in R$, choose a lowering operator $F_\beta \in \mathfrak g_{-\beta}$.\footnote{The root space $\mathfrak g_{-\beta}$ is the one-dimensional space of weight $-\beta$ in the adjoint representation.  Any nonzero element is a lowering operator, meaning that it sends a vector of weight $\eta$ to one of weight $\eta-\beta$ or to zero.}  On the symmetric square, it acts as a derivation: $F_\beta(vw)=(F_\beta v)w+v(F_\beta w)$.  The orthogonality of the roots in $R$ implies that the operators commute, so the product
\[
 z_\pi^Q:=\left(\prod_{\beta\in R}F_\beta\right)u_\mu^2\in V^{2\lambda}
\]
is well defined and has weight $2\mu-\sum_{\beta\in R}\beta=\mu+\nu$.  

By the Leibniz rule, $z_\pi^Q$ is a linear combination of products $u_{\mu_S}u_{\mu_{R\setminus S}}$ associated to pairs of opposite vertices of $C$.  The coefficient of $u_\mu u_\nu$ is nonzero, and we define $v_\pi^Q$ as $z_\pi^Q$ divided by this coefficient so that the coefficient of $u_{\mu} u_{\nu}$ in $v_\pi^Q$ is $1$.  When $\mu=\nu$, the set $R$ is empty, and $v_\pi^Q=u_\mu^2$.

Set $\mathcal B^Q := \{v_\pi^Q : \pi \in \RPP_2(P)\}$.  We prove that $\mathcal B^Q$ has the following properties:
\begin{itemize}
\item $\mathcal B^Q$ is a \emph{weight basis} of $V^{2\lambda}$, with each $v_\pi^Q$ of weight $\mu+\nu$.
\item Every $v_\pi^Q$ is a linear combination of distinct products $u_\eta u_\theta$ (with $\eta,\theta\in W\lambda$), each appearing with coefficient $1$ or $-1$.
\item Every integer linear combination of the products $u_\eta u_\theta$ that lies in $V^{2\lambda}$ also has integer coefficients in $\mathcal B^Q$.  In other words, $\mathcal B^Q$ is a $\Z$-basis of the lattice
\[
 V^{2\lambda}\cap\Sym^2V^\lambda_\Z,
 \qquad
 V^\lambda_\Z:=\bigoplus_{\eta\in W\lambda}\Z u_\eta.
\]
\end{itemize}

In fact, for each pair of weights $\mu\geq\nu$ of $V^\lambda$, we may choose any set of pairwise-orthogonal positive roots summing to $\mu-\nu$, independently of the choices for other pairs.  Applying the same construction with these sets in place of the $R_Q(\mu-\nu)$ still gives a basis with all the preceding properties (Theorem~\ref{thm:arbitrary-basis}).

What distinguishes the quiver choices is their compatibility with the generators of the Tits group.  At a \emph{source} or \emph{sink} --- a vertex whose incident arrows all point out or all point in --- the Tits representative of the corresponding simple reflection carries $\mathcal B^Q$, up to signs, to the quiver basis for the orientation obtained by reversing those arrows.  The product of these representatives in an order compatible with $Q$ takes $\mathcal B^Q$ back to itself --- and lifts Coxeter-motion, again up to signs.  

\subsection{Compatibility with Coxeter-motion}\label{subsec:intro-results}
\label{subsec:intro-compatibility}

Let $G$ be the simply connected Lie group with Lie algebra $\mathfrak g$, and $T\subset G$ a maximal torus with Lie algebra $\mathfrak h$.  Let $e_i\in\mathfrak g_{\alpha_i}$ and $f_i\in\mathfrak g_{-\alpha_i}$ be the Chevalley generators defining the standard weight basis $\{u_\mu\}$, and set
\[
 n_i:=\exp(e_i)\exp(-f_i)\exp(e_i)\in N_G(T).
\]
These elements generate the \emph{Tits group}.\footnote{See Tits~\cite{Tits} for the original construction, and Adams--Vogan~\cite[\S5]{AdamsVogan} and Adams--He~\cite[\S2]{AdamsHe} for modern accounts.  The generators satisfy $n_i^2=\alpha_i^\vee(-1)$.}  Under the quotient map $N_G(T)\rightarrow N_G(T)/T \simeq W$, each $n_i$ maps to the corresponding simple reflection $s_i$.  Furthermore, the $n_i$ obey the braid relations.  Thus, for any reduced expression $w=s_{i_t}\cdots s_{i_1}$, the product $\dot w=n_{i_t}\cdots n_{i_1} \in N_G(T)$ is a well-defined representative of $w$.\footnote{For $\mathfrak g=\mathfrak{sl}_n$, the standard Chevalley generators and $c=(1\,2\,\cdots\,n)=s_1\cdots s_{n-1}$ give
\[
 \dot c=g\,\operatorname{diag}(-1,\ldots,-1,1),
\]
where $g$ is the long-cycle permutation matrix used in Section~\ref{subsec:csp-history}.  The actions of $\dot c$ and $g$ on weight vectors differ only by signs.}

Denote by $\tau_i$ the composition of the toggles at all elements of $P$ labeled $i$, whether ordinary toggles on $J(P)$ or piecewise-linear toggles on $\RPP_m(P)$.  For a Coxeter element $c = s_{i_r} \cdots s_{i_1}$, Coxeter-motion by $c$ is given by $\gamma_c = \tau_{i_r} \cdots \tau_{i_1}$.  Rush--Shi proved that the bijection $\wt\colon J(P)\to W\lambda$ intertwines $\gamma_c$ and $c$ by showing that it intertwines each label toggle $\tau_i$ with the corresponding simple reflection $s_i$~\cite[Theorem~6.3(ii)]{RS}:
\[
 s_i\wt(I) = \wt(\tau_i I)\qquad(I\in J(P)).
\]
Consequently, the Tits generators act on the standard weight basis\footnote{We index the weight basis by order ideals via $u_I:=u_{\wt(I)}$.} by
\[
 n_i u_I = \pm u_{\tau_i I}\qquad(I \in J(P)).
\]

We prove a height-two counterpart of this identity.  Our local lifting theorem establishes not that a Tits generator permutes a quiver basis up to signs, but rather that, at a source or sink, it carries one quiver basis to another, up to signs, and lifts the piecewise-linear label toggle on the underlying RPPs.  

For a source or sink $i$ of $Q$, let $\xi_i(Q)$ be the orientation obtained by reversing the arrows incident to $i$.  When there is no ambiguity, we write $Q^i$ for $\xi_i(Q)$.  

\begin{theorem}[Local lifting]\label{thm:intro-local}
For every source or sink $i$ of $Q$, 
\begin{equation}\label{eq:intro-local}
 n_i v^Q_\pi=\pm v^{Q^i}_{\tau_i(\pi)} \qquad (\pi \in \RPP_2(P)).
\end{equation}
The sign is negative precisely when exactly one of the two endpoint pairings $\langle\mu,\alpha_i^\vee\rangle$ and $\langle\nu,\alpha_i^\vee\rangle$ equals $1$, and positive otherwise.
\end{theorem}

For an orientation $Q$ of the Dynkin diagram of $\mathfrak g$, let $\mathcal I = (i_1, \ldots, i_r)$ be an ordering of the vertices consistent with $Q$, i.e., such that, for each directed edge $j \rightarrow j'$ in $Q$, $j$ precedes $j'$ in $\mathcal I$.  Set $c_Q := s_{i_r} \cdots s_{i_1}$.  Different consistent orderings correspond to different reduced expressions for the same Coxeter element, so the association $Q \mapsto c_Q$ is well defined.\footnote{Two orderings yield the same Coxeter element if and only if they give the same relative order to every pair of vertices joined by an edge; see Shi~\cite[Proposition~1.3]{ShiCoxeter}.}  In fact, it is bijective, for from the reduced expression $s_{i_r} \cdots s_{i_1}$, we recover the ordering $\mathcal I$, which is consistent with precisely one orientation $Q$.  

Given a Coxeter element $c = s_{i_r} \cdots s_{i_1}$, we consider the orientation $Q$ such that $c_Q = c$.  Starting from $Q$, reversing the arrows at the vertices in the order $i_1,\ldots,i_r$ reverses each edge twice, so the final orientation is again $Q$.  Furthermore, each vertex is a source immediately before its arrows are reversed: We see inductively that $i_k$ is a source of $\xi_{i_{k-1}} \cdots \xi_{i_1}(Q)$ for all $1 \leq k \leq r$.  

Applying the local lifting theorem at each step, we find that $n_{i_k}$ carries the quiver basis for $\xi_{i_{k-1}} \cdots \xi_{i_1}(Q)$ to the quiver basis for $\xi_{i_k} \cdots \xi_{i_1}(Q)$, up to signs, and lifts the label toggle $\tau_{i_k}$.  Therefore, $\dot c_Q=n_{i_r}\cdots n_{i_1}$ preserves $\mathcal B^Q$ up to signs and lifts Coxeter-motion on RPPs, as stated in our main theorem.\footnote{The first assertion of the theorem follows directly from local lifting.  The second assertion, that the basis vectors can be rescaled so that $\dot c_Q$ lifts $\gamma_{c_Q}$ up to a single scalar, requires an additional argument; see the end of Section~\ref{sec:coxeter}.}

\begin{theorem}[Coxeter lifting]\label{thm:intro-coxeter}
For every orientation $Q$, the Tits representative $\dot c_Q$ lifts Coxeter-motion on $\mathcal B^Q$ up to signs:
\[
 \dot c_Q v^Q_\pi=\pm v^Q_{\gamma_{c_Q}(\pi)}
 \qquad(\pi\in\RPP_2(P)).
\]
Let $h$ be the Coxeter number of $\mathfrak{g}$, and set $\zeta:=\exp(2\pi i/h)$.
Then there exists a basis $\{\widehat v_\pi^Q:\pi\in\RPP_2(P)\}$, with each $\widehat v_\pi^Q$ a nonzero scalar multiple of $v_\pi^Q$, such that
\begin{equation}\label{eq:intro-coxeter}
 \zeta^{|P|}\dot c_Q\,\widehat v^Q_\pi
 =\widehat v^Q_{\gamma_{c_Q}(\pi)}
 \qquad(\pi\in\RPP_2(P)).
\end{equation}
\end{theorem}

The scalar $\zeta^{|P|}$ accounts for the grading shift between the character of $V^{2\lambda}$ and $F_{P,2}(q)$.  Taking traces and using Kostant's Coxeter-element theorem\footnote{For Kostant's theorem, see Prasad~\cite{Prasad}.} together with the conjugacy of Coxeter-motion and rowmotion, we obtain the following cyclic sieving result.  

\begin{corollary}\label{cor:intro-csp}
For every minuscule poset $P$ and every Coxeter element $c$, piecewise-linear Coxeter-motion and rowmotion on $\RPP_2(P)$ have order $h$ and exhibit the cyclic sieving phenomenon with respect to $F_{P,2}(q)$.  In particular,
\[
 \#\Fix(\gamma_c^d)=\#\Fix(\row^d)=F_{P,2}(\zeta^d)
 \qquad(d\in\Z).
\]
\end{corollary}

\subsection{Comparison with other bases}\label{subsec:intro-comparison}

In type~$A$, for the orientation $Q:1\to2\to\cdots\to n-1$, the quiver basis $\mathcal{B}^Q$ agrees with the canonical basis up to rescaling.  By duality, $\mathcal{B}^Q$ recovers the construction of Lam~\cite{Lam}, who realized the dual canonical basis via noncrossing matchings.

Fix $\pi\in\RPP_2([k]\times[n-k])$, and let $\mu\geq\nu$ be the weights determined by its two threshold ideals (cf. Section~\ref{subsec:intro-construction}).  Let $T \in S_2$ be the tableau such that $\pi_T = \pi$, and note that the columns $T^{(1)}$ and $T^{(2)}$ of $T$, viewed as $\mathfrak{sl}_n$ weights, coincide with $\mu$ and $\nu$, respectively.  Lam~\cite{Lam} associates to $T$ a noncrossing matching $M(T)$ between the sets $T^{(1)}\setminus T^{(2)}$ and $T^{(2)}\setminus T^{(1)}$.  We show that
\[
 \{\alpha_p+\cdots+\alpha_{q-1}:\{p,q\}\in M(T), \ p<q\}=R_Q(\mu-\nu).
\]
Section~\ref{sec:comparison} proves the identification, and Figure~\ref{fig:lam-matching} illustrates the tableau, matching, and corresponding cube.

For $k=1$ or $k=n-1$, every weight space of $V^{2\omega_k}$ is one-dimensional, so every Tits representative $\dot w \in N_G(T)$ permutes the canonical basis up to scalars.  For $2\leq k\leq n-2$, on the other hand, the long cycle and its inverse are the only Coxeter elements whose Tits representatives permute the canonical basis of $V^{2\omega_k}$ up to scalars (Proposition~\ref{prop:type-a-canonical-symmetries}).  For every other Coxeter element, the corresponding quiver basis cannot be obtained by rescaling the canonical basis, as illustrated in Example~\ref{ex:A3}.  

In types~$D$ and~$E$, no Tits representative of any Coxeter element permutes the canonical basis of $V^{2\lambda}$ up to scalars.  The same holds for the dual canonical basis of $(V^{2\lambda})^*$ (Theorem~\ref{thm:canonical-obstruction}).  Thus, in these types, bases compatible with Coxeter-motion, such as our quiver bases, must differ from the canonical basis by more than rescaling.

Garver, Patrias, and Thomas~\cite[Theorem~1.6]{GPT} present a bijection between RPPs of unrestricted height on a minuscule poset $P$ and isomorphism classes of quiver representations whose indecomposable summands are nonzero at the vertex corresponding to the minuscule weight $\lambda$.  They extend their construction to relate formal RPPs (with entries such as $\infty - a$) to root-category objects.  Under this relationship, given a source vertex $i$ of a quiver $Q$, the piecewise-linear label toggle $\tau_i$ corresponds to applying the derived reflection functor $\widetilde R_i^{-}\colon D^b(\operatorname{rep}Q)\longrightarrow D^b(\operatorname{rep}\xi_i(Q))$~\cite[Theorem~1.10]{GPT}.\footnote{For a representation $X$ of $Q$, ordinary reflection at a source $i$ sends $X_i$ to the cokernel of $X_i\to\bigoplus_{i\to j}X_j$, with the induced maps along the reversed arrows.  The derived version retains simple summands supported only at $i$ as shifted simple representations.  The root category is $D^b(\operatorname{rep}Q)/[2]$, and the derived reflection functor descends to the corresponding root categories~\cite[Sections~3.1 and~5.1--5.3]{GPT}.}  

This correspondence is the linchpin of their proposed proof that $\gamma_c^h=1$ at arbitrary height, which would then follow from the order $h$ of Auslander--Reiten translation on isomorphism classes in the root category.  Unfortunately, the passage from formal entries involving $\infty$ to bounded RPPs relies on~\cite[Lemma~5.7]{GPT}, which does not hold as stated.\footnote{In their order-reversing convention, take the diamond $a<b,c<d$ with entries $(\rho(a),\rho(b),\rho(c),\rho(d))=(\infty,2,\infty-2,0)$.  Substituting $\infty=3$ gives the valid bounded RPP $(3,2,1,0)$, meeting the lemma's hypothesis.  At the maximum $d$, however, toggling first gives $\min(2,\infty-2)=2$, whereas substituting first and then toggling gives $\min(2,1)=1$.  Moreover, substituting after the formal toggle gives $(3,2,1,2)$, which is not order-reversing.}  

The periodicity identity nonetheless follows independently from tropicalizing the birational result of Okada~\cite[Theorem~3(a)]{Okada}.  Okada's proof proceeds case by case, so our proof that Coxeter-motion and rowmotion have order $h$ on $\RPP_2(P)$ is the first uniform one.  Furthermore, even if the Garver--Patrias--Thomas argument were repaired, it would establish uniformly that $\gamma_c$ has order $h$ at arbitrary height, but not that it exhibits cyclic sieving.  

For every height $m\geq1$, Strayer~\cite[Theorem~9.5]{Strayer} constructs combinatorially a weight basis of $V^{m\lambda}$ indexed by $\RPP_m(P)$.  His basis is not, however, compatible with Coxeter-motion in general: At height two, the Tits representative in Example~\ref{ex:D4} permutes our quiver basis up to signs but does not permute Strayer's basis up to scalars.  For general $m$, constructing bases with the compatibility proved herein remains open.  

Section~\ref{sec:background} reviews the background, including the minuscule weight and toggle conventions.  Sections~\ref{sec:bases} and~\ref{sec:reflections} construct the quiver bases and prove the local lifting theorem; Sections~\ref{sec:coxeter} and~\ref{sec:csp} establish the main theorem and cyclic sieving corollary.  Sections~\ref{sec:comparison} and~\ref{sec:canonical-obstruction} show the canonical-basis comparison and obstruction.  Section~\ref{sec:higher-powers} discusses the open problem for general $m$.  The classical convention checks and the longer obstruction calculations appear in the appendices.

\section{Minuscule heaps and piecewise-linear toggles}\label{sec:background}

Let $\mathfrak g$ be a simply laced complex simple Lie algebra and $G$ its simply connected complex Lie group.  Fix a Cartan subalgebra $\mathfrak h\subset\mathfrak g$, and let $T\subset G$ be the maximal torus with Lie algebra $\mathfrak h$.  Let $\Phi\subset\mathfrak h^*$ be the root system of $(\mathfrak g,\mathfrak h)$.  Choose positive roots $\Phi^+$ and index the corresponding simple roots as $\{\alpha_i:i\in I\}$.  Let $W:=N_G(T)/T$ be the Weyl group, with simple reflections $s_i$ for $i\in I$ and longest element $w_0$.  Set $r:=|I|$, and let $h$ be the Coxeter number.  Write $\omega_i$ for the corresponding fundamental weights, and $\rho$ and $\rho^\vee$ for the half-sums of positive roots and positive coroots.  We normalize the invariant inner product by $(\alpha,\alpha)=2$ for every root.  Set $L:=\Z\Phi$ and $L_+:=\sum_i\Z_{\geq0}\alpha_i$.  We use root order on weights: $\mu\geq\nu$ means $\mu-\nu\in L_+$.  The notation $i\sim j$ means that $i$ and $j$ are adjacent vertices of the Dynkin diagram.

\subsection{The minuscule representation}
A nonzero dominant weight $\lambda$ is minuscule if the weights of the irreducible representation $V^\lambda$ form a single Weyl group orbit.  Fix such a weight $\lambda$ and set $\Omega:=W\lambda$.  Each weight space is one-dimensional, and
\[
 (\mu,\beta)\in\{-1,0,1\}
 \qquad(\mu\in\Omega,\ \beta\in\Phi).
\]
Choose Chevalley generators $e_i\in\mathfrak g_{\alpha_i}$ and $f_i\in\mathfrak g_{-\alpha_i}$, and fix a highest-weight vector $u_\lambda$.  The simple lowering operators determine the standard weight basis $\{u_\mu:\mu\in\Omega\}$, normalized by
\[
 f_i u_\mu=u_{\mu-\alpha_i}\quad\text{if }(\mu,\alpha_i)=1,
 \qquad
 e_i u_\mu=u_{\mu+\alpha_i}\quad\text{if }(\mu,\alpha_i)=-1,
\]
with all other actions zero.  This is the classical specialization of the global basis of a minuscule representation; see Geck~\cite{Geck}, Definition~2.2 and Proposition~2.9.  Denote its integral span by $V_\Z^\lambda$.  The Tits representative $n_i:=\exp(e_i)\exp(-f_i)\exp(e_i)$ acts on this basis by
\begin{equation}\label{eq:standard-tits}
 n_i u_\mu=\eps_i(\mu)u_{s_i\mu},\qquad
 \eps_i(\mu):=
 \begin{cases}
 -1&(\mu,\alpha_i)=1,\\
 1&(\mu,\alpha_i)=0\text{ or }-1.
 \end{cases}
\end{equation}
Indeed, this is the usual action on each two-dimensional simple-root string, and $n_i$ fixes the weight vectors on which both $e_i$ and $f_i$ vanish.

The Cartan square $V^{2\lambda}$ is the submodule of $\Sym^2V^\lambda$ generated by $u_\lambda^2$.  Every square $u_\mu^2$, for $\mu\in\Omega$, belongs to it: a product of Tits representatives taking $\lambda$ to $\mu$ takes $u_\lambda^2$ to $u_\mu^2$.  We will construct the quiver basis by applying root operators to these squares.

\subsection{The heap and its thresholds}
Proctor~\cite{Proctor} shows that $\Omega$, ordered oppositely to root order, is a distributive lattice with minimal element $\lambda$.  Let $P$ be its poset of join-irreducibles.  Under the resulting identification $J(P)\cong\Omega$, adjoining an element $p$ to an order ideal subtracts a simple root from its weight.  This root depends only on $p$; its index is the label $\ell(p)\in I$.  Thus the weight of an ideal $A$ is
\begin{equation}\label{eq:ideal-weight}
 \wt(A)=\lambda-\sum_{p\in A}\alpha_{\ell(p)}.
\end{equation}
The labeled poset is the \emph{minuscule heap} of Stembridge~\cite{StembridgeFC,StembridgeMin}; its underlying poset is the minuscule poset.  Our order convention follows Rush and Wang~\cite{RW} and Rush~\cite{RushCDE}.  In particular,
\[
 A\subseteq B\quad\Longleftrightarrow\quad\wt(A)\geq\wt(B).
\]

The full ideal has weight $w_0\lambda$.  Since $\langle\alpha_i,\rho^\vee\rangle=1$ for each simple root, equation~\eqref{eq:ideal-weight} gives
\begin{equation}\label{eq:heap-size}
 |P|=\langle\lambda-w_0\lambda,\rho^\vee\rangle
 =\langle2\lambda,\rho^\vee\rangle,
\end{equation}
where the second equality uses $w_0\rho^\vee=-\rho^\vee$.

A reverse plane partition of height two is an order-preserving map $\pi:P\to\{0,1,2\}$.  Its threshold ideals
\[
 A:=\{p\in P:\pi(p)<1\},\qquad
 B:=\{p\in P:\pi(p)<2\}
\]
satisfy $A\subseteq B$.  Conversely, these ideals recover $\pi$ by the formula $\pi=2-\mathbf1_A-\mathbf1_B$.  Thus $\RPP_2(P)$ is in bijection with the comparable pairs
\[
 (\mu,\nu):=(\wt(A),\wt(B)),\qquad\mu\geq\nu,
\]
of weights of $V^\lambda$.  For such a pair, write
\begin{equation}\label{eq:mu-delta}
 \xi:=\mu+\nu,\qquad\delta:=\mu-\nu\in L_+.
\end{equation}
The vector indexed by $\pi$ will have weight $\xi$; its construction will depend on the difference $\delta$.

The number of comparable pairs with sum $\xi$ is the dimension of the corresponding weight space of the Cartan square:
\begin{equation}\label{eq:dimension}
 \dim V^{2\lambda}_\xi
 =\#\{(\mu,\nu)\in\Omega^2:\mu\geq\nu,\ \mu+\nu=\xi\}.
\end{equation}
Indeed, the degree-two standard monomials in the minuscule homogeneous coordinate ring form a basis of $(V^{2\lambda})^*$, indexed by comparable pairs $(\mu,\nu)$ and having weights $-(\mu+\nu)$; see Seshadri~\cite{Seshadri}.  Dualizing gives~\eqref{eq:dimension}.  Consequently, to construct a basis it will suffice to find linearly independent vectors with these indices and weights.

\subsection{Toggles}
For a positive integer $m$ and $p\in P$, the piecewise-linear toggle $t_p$ on $\RPP_m(P)$ fixes all entries except $\pi(p)$ and reflects that entry in the interval permitted by its neighbors:
\begin{equation}\label{eq:toggle}
 (t_p\pi)(p)
 :=\max_{q\lessdot p}\pi(q)+\min_{p\lessdot q}\pi(q)-\pi(p).
\end{equation}
The maximum is $0$ when $p$ is minimal, and the minimum is $m$ when $p$ is maximal.  Thus $t_p$ is an involution of $\RPP_m(P)$.  Two toggles commute unless their elements are joined by a cover.  Every cover in a minuscule heap joins elements with adjacent Dynkin labels.  Elements with the same label form a chain and are never joined by a cover, so the \emph{label toggle}
\[
 \tau_i:=\prod_{\ell(p)=i}t_p
\]
is independent of the order of its factors.  For a Coxeter element $c=s_{i_r}\cdots s_{i_1}$, we define piecewise-linear Coxeter-motion by $\gamma_c:=\tau_{i_r}\cdots\tau_{i_1}$, with the rightmost factor acting first.  Piecewise-linear rowmotion applies the individual toggles from maximal to minimal elements, in any reverse linear extension of $P$.

Although $\RPP_2(P)$ may also be identified with the order ideals of $P\times[2]$, the action here is piecewise-linear rowmotion on $P$.  It differs from ordinary order-ideal rowmotion on $P\times[2]$: for a one-element $P$, these actions have orders two and three, respectively.

At height one, Rush and Shi~\cite{RS}, Theorem~6.3(ii), identify the label toggle $\tau_i$ with the simple reflection $s_i$ under~\eqref{eq:ideal-weight}.  We now return to height two, where the two threshold ideals need to remain nested.  The following formula describes how this requirement changes the action on the two weights.

\begin{lemma}\label{lem:cutoff}
For a comparable pair $(\mu,\nu)$ with difference $\delta$,
\begin{equation}\label{eq:cutoff}
 \tau_i(\mu,\nu)=
 \begin{cases}
 (s_i\mu,s_i\nu)&s_i\delta\in L_+,\\
 (\mu,\nu)&s_i\delta\notin L_+.
 \end{cases}
\end{equation}
Writing $\delta=\sum_jd_j\alpha_j$, the second case occurs precisely when
\begin{equation}\label{eq:wall}
 d_i=1,\qquad d_j=0\ (j\sim i),\qquad
 (\mu,\alpha_i)=1,\quad(\nu,\alpha_i)=-1.
\end{equation}
In this case $(\xi,\alpha_i)=0$.
\end{lemma}
\begin{proof}
Let $A\subseteq B$ be the threshold ideals of $\pi$.  Each ideal has at most one toggleable element of label $i$.  Apply the height-one label toggle to both ideals.  If an element absent from both ideals is addable to $A$, it is also addable to $B$.  Dually, if an element present in both is removable from $B$, it is also removable from $A$.  Thus the two toggled ideals can fail to be nested only when the same element $p\in B\setminus A$ is added to $A$ and removed from $B$.

In this case $\pi(p)=1$.  Addability to $A$ makes every lower neighbor have value $0$, and removability from $B$ makes every upper neighbor have value $2$.  The interval in~\eqref{eq:toggle} is therefore $[0,2]$, so the piecewise-linear toggle fixes $\pi(p)$.  The other entries of label $i$ are fixed as well, since neither threshold ideal has another toggleable element of that label.  In every other case, the changes in threshold membership prescribed by~\eqref{eq:toggle} agree with the height-one toggles.  Since the reflected ideals are nested exactly when $s_i\mu\geq s_i\nu$, this proves~\eqref{eq:cutoff}.

Reflection changes only the coefficient of $\alpha_i$ in $\delta$.  Its new coefficient is
\[
 d_i-(\delta,\alpha_i),\qquad
 (\delta,\alpha_i)=2d_i-\sum_{j\sim i}d_j\leq2,
\]
where the inequality follows from the minuscule bounds for $\mu$ and $\nu$.  If $d_i=0$, the new coefficient is $\sum_{j\sim i}d_j\geq0$; if $d_i\geq2$, it is again nonnegative.  A negative coefficient therefore forces $d_i=1$ and $(\delta,\alpha_i)=2$.  The displayed formula then gives $d_j=0$ for every neighbor $j$ of $i$, and the minuscule bounds give $(\mu,\alpha_i)=1$ and $(\nu,\alpha_i)=-1$.  Conversely, these conditions make the new coefficient $-1$.  They also give $(\xi,\alpha_i)=(\mu+\nu,\alpha_i)=0$.
\end{proof}

\begin{remark}\label{rem:toggle-braid}
At height two, label toggles need not satisfy the braid relations.  For $V^{\omega_2}$ in type~$A_3$, the minuscule heap is a diamond whose bottom, left, right, and top elements have labels $2,1,3,2$.  Write a filling in that order, and set
\[
 a:=(0,1,1,2),\qquad b:=(1,1,1,1).
\]
The toggle $\tau_1$ fixes both fillings, whereas $\tau_2$ interchanges them.  Hence
\[
 \tau_1\tau_2\tau_1(a)=b
 \ne a=\tau_2\tau_1\tau_2(a).
\]
Nevertheless, reduced expressions for a fixed Coxeter element differ only by interchanging adjacent factors whose Dynkin vertices are nonadjacent.  The corresponding label toggles commute, so $\gamma_c$ is well defined.
\end{remark}

\section{Constructing the quiver bases}\label{sec:bases}

For each comparable pair $\mu\geq\nu$, we will construct a vector of weight $\xi=\mu+\nu$ by applying commuting root operators to $u_\mu^2$.  This requires a decomposition of $\delta=\mu-\nu$ into positive orthogonal roots.  An orientation $Q$ supplies such a decomposition through the indecomposable summands of a generic quiver representation.

\subsection{The generic decomposition}
Fix an orientation $Q$ of the Dynkin diagram.  A representation $M$ of $Q$ assigns a finite-dimensional complex vector space $M_i$ to each vertex and a linear map $M_i\to M_j$ to each arrow $i\to j$.  Its dimension vector is $\dim M:=\sum_i(\dim M_i)\alpha_i$.  For vectors $x=\sum_i x_i\alpha_i$ and $y=\sum_i y_i\alpha_i$, the Euler form of $Q$ is
\begin{equation}\label{eq:euler}
 E_Q(x,y):=\sum_i x_iy_i-\sum_{i\longrightarrow j}x_iy_j,
 \qquad E_Q(x,y)+E_Q(y,x)=(x,y).
\end{equation}
For representations $M,N$ of $Q$, it computes the difference between the dimensions of the homomorphism and extension spaces:
\begin{equation}\label{eq:hom-ext}
 E_Q(\dim M,\dim N)
 =\dim\Hom(M,N)-\dim\Ext^1(M,N).
\end{equation}

By Gabriel's theorem~\cite{Gabriel}, the indecomposable representations of $Q$ are indexed by the positive roots: write $M_\beta$ for the indecomposable of dimension vector $\beta$.  Each has one-dimensional endomorphism space and no self-extensions.  For a fixed dimension vector $d=\sum_i d_i\alpha_i$, the representation variety
\[
 \operatorname{Rep}_Q(d)
 :=\prod_{i\longrightarrow j}\Hom(\C^{d_i},\C^{d_j})
\]
is irreducible.  The group $\prod_i\operatorname{GL}_{d_i}(\C)$ acts by changes of bases at the vertices, and Gabriel's theorem implies that this action has finitely many orbits.  Exactly one orbit is open.  A representation in this orbit is \emph{generic}.

For a representation $M$ of dimension vector $d$, write $\mathcal O_M$ for its orbit.  Its stabilizer has dimension $\dim\Hom(M,M)$, so~\eqref{eq:hom-ext} gives
\[
 \operatorname{codim}_{\operatorname{Rep}_Q(d)}\mathcal O_M
 =\dim\Hom(M,M)-E_Q(d,d)
 =\dim\Ext^1(M,M).
\]
Thus a representation is generic exactly when it is \emph{rigid}, meaning that $\Ext^1(M,M)=0$.

For a difference $\delta=\mu-\nu$ of comparable weights in $\Omega$, rigidity forces more than the vanishing of extension spaces: no indecomposable summand repeats, and there are no nonzero homomorphisms between distinct summands.

\begin{proposition}\label{prop:generic}
Let $\mu,\nu\in\Omega$ with $\mu\geq\nu$.  The generic representation of $Q$ of dimension vector $\delta=\mu-\nu$ is a direct sum of pairwise nonisomorphic indecomposables, each occurring once.  For distinct summands $M_\beta$ and $M_\gamma$,
\[
 \Hom(M_\beta,M_\gamma)=0,\qquad
 \Ext^1(M_\beta,M_\gamma)=0.
\]
The set $R_Q(\delta)$ of their dimension vectors satisfies
\begin{equation}\label{eq:recording}
 \delta=\sum_{\beta\in R_Q(\delta)}\beta,
 \qquad
 E_Q(\beta,\gamma)=
 \begin{cases}
 1&\beta=\gamma,\\
 0&\beta\ne\gamma.
 \end{cases}
\end{equation}
In particular, distinct roots in $R_Q(\delta)$ are orthogonal.
\end{proposition}
\begin{proof}
Write the generic representation as $\bigoplus_jM_{\beta_j}^{\oplus m_j}$, with the $M_{\beta_j}$ pairwise nonisomorphic.  Rigidity gives $\Ext^1(M_{\beta_j},M_{\beta_k})=0$ for every $j,k$.  For $j\ne k$, equations~\eqref{eq:euler} and~\eqref{eq:hom-ext} therefore give
\[
 (\beta_j,\beta_k)
 =\dim\Hom(M_{\beta_j},M_{\beta_k})
  +\dim\Hom(M_{\beta_k},M_{\beta_j})\geq0.
\]
On the other hand, $\delta=\mu-\nu$ and the minuscule bounds imply
\[
 2\geq(\delta,\beta_j)
 =2m_j+\sum_{k\ne j}m_k(\beta_k,\beta_j).
\]
Since $m_j\geq1$ and every term in the sum is nonnegative, $m_j=1$ and $(\beta_k,\beta_j)=0$ for $k\ne j$.  Both homomorphism spaces in the preceding display must consequently vanish.  The Euler formula now gives~\eqref{eq:recording}.  Uniqueness of the open orbit and the Krull--Schmidt theorem make $R_Q(\delta)$ well defined.  When $\delta=0$, the generic representation is zero and $R_Q(0)=\varnothing$.
\end{proof}

\subsection{The cube vector}
Let $R$ be any set of positive pairwise orthogonal roots with sum $\delta=\mu-\nu$.  We call $R$ an \emph{orthogonal decomposition} of $\delta$; Proposition~\ref{prop:generic} supplies the decomposition $R_Q(\delta)$ for every orientation $Q$.  For $\beta\in R$, orthogonality gives $(\delta,\beta)=2$, so the minuscule bounds force
\[
 (\mu,\beta)=1,\qquad(\nu,\beta)=-1.
\]
Thus $s_\beta\mu=\mu-\beta$.  Since the reflections in the roots of $R$ commute, for every subset $S\subseteq R$ we have
\begin{equation}\label{eq:cube}
 \mu_S=\left(\prod_{\beta\in S}s_\beta\right)\mu
 =\mu-\sum_{\beta\in S}\beta\in\Omega.
\end{equation}
These weights are the vertices of a cube with endpoints $\mu_\varnothing=\mu$ and $\mu_R=\nu$.  Opposite vertices satisfy $\mu_S+\mu_{R\setminus S}=\xi$.

Since every root has squared length $2$, neither the sum nor the difference of two orthogonal roots is a root.  Their root $\mathfrak{sl}_2$ subalgebras therefore commute.  Choose a lowering operator $F_\beta\in\mathfrak g_{-\beta}$ for each $\beta\in R$, and set
\begin{equation}\label{eq:z}
 z_{\mu,\nu;R}
 :=\left(\prod_{\beta\in R}F_\beta\right)u_\mu^2
 \in V^{2\lambda}_\xi.
\end{equation}
For $S\subseteq R$, write $w_S:=(\prod_{\beta\in S}F_\beta)u_\mu$.  Each $w_S$ is a nonzero multiple of $u_{\mu_S}$: before an unused root operator $F_\beta$ is applied, the current weight still has pairing $1$ with $\beta$.  Each root operator acts on a symmetric product by the derivation rule, so
\begin{equation}\label{eq:expansion}
 z_{\mu,\nu;R}
 =\sum_{S\subseteq R}w_Sw_{R\setminus S}.
\end{equation}
For example, if $R=\{\beta,\gamma\}$, this is the sum of the products associated with the two pairs of opposite vertices:
\[
 z_{\mu,\nu;R}
 =2\bigl(u_\mu F_\beta F_\gamma u_\mu
 +(F_\beta u_\mu)(F_\gamma u_\mu)\bigr).
\]

We use the symmetric-square monomial basis consisting of $u_\eta u_\theta$ for unordered pairs $\{\eta,\theta\}$ of weights, allowing $\eta=\theta$.  The roots in $R$ are linearly independent.  Consequently, two subsets give the same monomial in~\eqref{eq:expansion} only when they are equal or complementary.  Complementary subsets contribute identical terms, so there is no cancellation.  In particular, the coefficient of $u_\mu u_\nu$ is nonzero.  Rescaling any $F_\beta$ rescales the whole vector, so its line depends only on $(\mu,\nu;R)$.  Let $v_{\mu,\nu;R}$ be the unique vector on this line whose coefficient at $u_\mu u_\nu$ is $1$.

\begin{theorem}\label{thm:arbitrary-basis}
Choose an orthogonal decomposition $R(\mu,\nu)$ of $\mu-\nu$ for every comparable pair $\mu\geq\nu$ in $\Omega$.  Then the vectors
\[
 v_{\mu,\nu;R(\mu,\nu)}\qquad(\mu,\nu\in\Omega,\ \mu\geq\nu)
\]
form a weight basis of $V^{2\lambda}$, with $v_{\mu,\nu;R(\mu,\nu)}$ of weight $\mu+\nu$.  Their coefficients in the standard symmetric-square monomial basis belong to $\{0,1,-1\}$, and they form a $\Z$-basis of the saturated lattice
\begin{equation}\label{eq:saturated}
 \LL:=V^{2\lambda}\cap\Sym^2V_\Z^\lambda.
\end{equation}
\end{theorem}
\begin{proof}
Fix a weight $\xi$.  For each comparable pair $\mu\geq\nu$ with sum $\xi$, consider the coefficient functional that extracts the coefficient of $u_\mu u_\nu$.  Evaluate these functionals on the proposed vectors of weight $\xi$.  By~\eqref{eq:dimension}, the resulting matrix is square, with its number of rows equal to $\dim V^{2\lambda}_\xi$.

Suppose that the monomial $u_{\mu'}u_{\nu'}$ occurs in $v_{\mu,\nu;R(\mu,\nu)}$ and that $\mu'\geq\nu'$.  Its upper endpoint $\mu'$ is one of the cube vertices $\mu_S$, so
\[
 \mu-\mu'=\sum_{\beta\in S}\beta\in L_+.
\]
If $\mu'=\mu$, then $S=\varnothing$ and $(\mu',\nu')=(\mu,\nu)$.  Order the comparable pairs by a linear extension of root order on their upper endpoints, placing $\mu'$ before $\mu$ whenever $\mu'<\mu$.  The coefficient matrix is then upper triangular, with diagonal entries $1$ by our normalization.  It is invertible, so the vectors are linearly independent.  Equation~\eqref{eq:dimension} proves that they form a basis of $V^{2\lambda}_\xi$.

To determine the coefficients, choose each $F_\beta$ to be a Tits conjugate of a simple lowering operator.  By~\eqref{eq:standard-tits}, Tits representatives permute the standard weight basis up to sign, so every nonzero coefficient of $F_\beta$ in that basis is $\pm1$.  Hence $w_S=\pm u_{\mu_S}$.  If $R$ is nonempty, each monomial in~\eqref{eq:expansion} occurs exactly twice, from complementary subsets, with the same sign.  Every nonzero coefficient of $z_{\mu,\nu;R}$ is therefore $2$ or $-2$.  Dividing by the coefficient at $u_\mu u_\nu$ gives coefficients $\pm1$ in $v_{\mu,\nu;R}$.  The empty decomposition gives $u_\mu^2$.  Thus all the basis vectors belong to $\LL$.

The coefficient matrix above is now an integer unitriangular matrix, so its inverse also has integer entries.  If $x\in\LL$ has weight $\xi$, its coefficients at the comparable monomials are integers.  Multiplying this coefficient vector by the inverse matrix expresses $x$ as an integer linear combination of the proposed basis vectors.  Finally, the ambient monomial basis is a weight basis, so every element of $\LL$ is the sum of its integral weight components.  This proves the lattice assertion.
\end{proof}

For the reverse plane partition $\pi$ corresponding to $(\mu,\nu)$, define
\[
 v^Q_\pi=v^Q_{\mu,\nu}
 :=v_{\mu,\nu;R_Q(\mu-\nu)}.
\]
The preceding theorem shows that
\[
 \mathcal B^Q:=\{v^Q_\pi:\pi\in\RPP_2(P)\}
\]
is a basis of $V^{2\lambda}$.  We call it the \emph{quiver basis} associated with $Q$.

The dimension of the cube also determines the number of monomials in a basis vector.  With $k:=|R_Q(\mu-\nu)|$, orthogonality and $(\mu,\mu)=(\nu,\nu)=(\lambda,\lambda)$ give
\begin{equation}\label{eq:support-size}
 k=\tfrac12(\mu-\nu,\mu-\nu)
 =(\lambda,\lambda)-(\mu,\nu).
\end{equation}
There is one monomial for each pair of opposite vertices, so $v^Q_\pi$ has $2^{k-1}$ monomials when $k>0$, and one when $k=0$.

\begin{remark}
The lattice~\eqref{eq:saturated} need not be the highest-weight lattice generated by divided powers from $u_\lambda^2$.  For the two-dimensional $\mathfrak{sl}_2$-module, set $u_+:=u_\lambda$ and $u_-:=f_1u_+$.  Then $f_1(u_+^2)=2u_+u_-$, whereas the saturated lattice contains $u_+u_-$.  The integral normalization in Theorem~\ref{thm:arbitrary-basis} is relative to the ambient symmetric square.
\end{remark}

\section{Lifting simple reflections}\label{sec:reflections}

We now prove that the Tits representative $n_i$ takes each quiver basis vector, up to sign, to the vector indexed by its label toggle whenever $i$ is a source or sink.  The target vector belongs to the basis for the reflected orientation.  Reflection functors describe the change in the indecomposable summands, and conjugation by $n_i$ gives the corresponding change in their root operators.

For a source or sink $i$ of $Q$, let $Q^i$ be obtained by reversing the arrows incident to $i$.  The reflection functor of Bernstein, Gel'fand, and Ponomarev~\cite{BGP}, Definition~1.1, leaves the vector spaces away from $i$ unchanged and replaces $M_i$ by
\[
 \begin{cases}
 \operatorname{coker}\!\left(M_i\longrightarrow\displaystyle\bigoplus_{i\to j}M_j\right)&\text{if $i$ is a source},\\[6pt]
 \ker\!\left(\displaystyle\bigoplus_{j\to i}M_j\longrightarrow M_i\right)&\text{if $i$ is a sink}.
 \end{cases}
\]
The maps on reversed arrows are induced by the inclusions into the direct sum and the projections from it.  The simple representation $S_i$, with a one-dimensional space at $i$ and zero spaces elsewhere, is sent to zero.  Every other indecomposable $M_\beta$ is sent to the indecomposable of dimension vector $s_i\beta$.  On the subcategories without $S_i$ as a direct summand, the reflection functors are inverse equivalences and preserve Hom spaces.  The identity
\begin{equation}\label{eq:euler-reflection}
 E_{Q^i}(s_i x,s_i y)=E_Q(x,y)
\end{equation}
then shows that they preserve the dimensions of Ext spaces as well.

\begin{proposition}\label{prop:transport}
Let $i$ be a source or sink of $Q$, let $\mu\geq\nu$ be weights of $V^\lambda$, and set $(\mu',\nu'):=\tau_i(\mu,\nu)$.  Then
\begin{equation}\label{eq:recording-transport}
 R_{Q^i}(\mu'-\nu')
 =\{\pos(s_i\beta):\beta\in R_Q(\mu-\nu)\},
\end{equation}
where $\pos$ chooses the positive root on a root line, and
\begin{equation}\label{eq:local-signed}
 n_i v^Q_{\mu,\nu}
 =\eps_i(\mu)\eps_i(\nu)v^{Q^i}_{\mu',\nu'}.
\end{equation}
\end{proposition}
\begin{proof}
Set $\delta:=\mu-\nu$.  Suppose first that $S_i$ occurs in the generic representation of dimension $\delta$.  For any representation $M$ of $Q$, we have $\dim\Hom(M,S_i)=(\dim M)_i$ if $i$ is a source, and $\dim\Hom(S_i,M)=(\dim M)_i$ if $i$ is a sink.  Since there are no nonzero homomorphisms between distinct summands, every other summand has zero coordinate at $i$.  Its root is orthogonal to $\alpha_i$, so its coordinates at all neighbors of $i$ are also zero.  Thus $\delta$ satisfies~\eqref{eq:wall}.  Conversely, these coordinate conditions split off one copy of $S_i$.  All remaining summands have zero spaces at $i$ and its neighbors and are unchanged by the orientation reflection.  Since $\pos(s_i\alpha_i)=\alpha_i$, this proves~\eqref{eq:recording-transport} in the fixed case of~\eqref{eq:cutoff}.

If $S_i$ does not occur, the reflection functor sends all summands to indecomposables with positive roots $s_i\beta$.  Their direct sum is rigid, hence generic of dimension $s_i\delta$.  This proves~\eqref{eq:recording-transport} in the other case.

We next apply the Tits representative to the cube vector.  When $s_i\delta\in L_+$, conjugation by $n_i$ carries $\mathfrak g_{-\beta}$ to $\mathfrak g_{-s_i\beta}$, and $n_i$ sends $u_\mu^2$ to $u_{s_i\mu}^2$.  Applying $n_i$ to~\eqref{eq:z} therefore gives a nonzero multiple of the cube vector indexed by $(\mu',\nu')$ for $Q^i$.  When $s_i\delta\notin L_+$, one of the roots is $\alpha_i$.  Applying the other root operators to $u_\mu^2$ gives a nonzero highest-weight vector of weight two for the $i$th $\mathfrak{sl}_2$, since these operators commute with that subalgebra.  Applying $F_{\alpha_i}$ then gives the middle-weight vector~\eqref{eq:z}, on which $n_i$ acts by $-1$.  Thus $n_i$ takes the original vector to a nonzero multiple of the required vector in both cases.

The endpoint normalization determines the scalar.  In the first case, $u_\mu u_\nu$ is the unique monomial mapping to $u_{s_i\mu}u_{s_i\nu}$, with coefficient $\eps_i(\mu)\eps_i(\nu)$ by~\eqref{eq:standard-tits}.  In the fixed case, the preceding $\mathfrak{sl}_2$ argument gives the scalar $-1=\eps_i(\mu)\eps_i(\nu)$.  Since the target vector has endpoint coefficient $1$, equation~\eqref{eq:local-signed} follows.  This proves Theorem~\ref{thm:intro-local}.
\end{proof}

\subsection{The reflect-and-peel algorithm}
Fix $\mu,\nu\in\Omega$ with $\mu\geq\nu$, and set $\delta:=\mu-\nu$.  The same reflections compute $R_Q(\delta)$ without constructing a quiver representation.  Choose a source order $i_1,\ldots,i_r$ for $Q$, meaning an ordering of the vertices consistent with its arrows, and set $c_Q:=s_{i_r}\cdots s_{i_1}$ and $\dot c_Q:=n_{i_r}\cdots n_{i_1}$.  These products depend only on $Q$, since any two source orders differ by interchanging consecutive nonadjacent vertices.

Keep $Q$ for the initial orientation, and repeat the chosen order: each vertex is a source when reached, and a complete sweep returns the orientation to $Q$.  Start with current orientation $Q$, $d:=\delta$, $w:=1\in W$, and an empty output.  At the current source $i$, perform one of the following steps:
\begin{enumerate}
\item If $d_i=1$ and $d_j=0$ for every neighbor $j$ of $i$, output $w^{-1}\alpha_i$ and replace $d$ by $d-\alpha_i$.
\item Otherwise, replace $d$ by $s_i d$.
\end{enumerate}
In both cases replace $w$ by $s_iw$, reverse the arrows incident to $i$ in the current orientation, and continue until $d=0$.  We call this the \emph{reflect-and-peel algorithm}.  A peeling step records and removes a simple-root summand.  The corresponding toggle in~\eqref{eq:cutoff} fixes the pair and retains that summand.

\begin{proposition}\label{prop:algorithm}
The reflect-and-peel algorithm terminates within $h$ complete sweeps and outputs precisely $R_Q(\delta)$.  Its output is therefore independent of the source order compatible with $Q$.
\end{proposition}
\begin{proof}
At each stage, $d$ remains a difference of comparable weights of $V^\lambda$.  Indeed, write the current difference as $d=\mu_0-\nu_0$.  At a reflection step the reflected weights remain comparable by~\eqref{eq:cutoff}; at a peeling step,
\[
 d-\alpha_i=(\mu_0-\alpha_i)-\nu_0,
 \qquad \mu_0-\alpha_i=s_i\mu_0\in\Omega
\]
by~\eqref{eq:wall}.

Track the indecomposable summands of the generic representation.  By Proposition~\ref{prop:transport}, the algorithm removes a summand exactly when its current root is $\alpha_i$.  Every other root is reflected and remains positive.  At a peeling step, the remaining summands avoid $i$ and its neighbors and are fixed by $s_i$.  Thus a surviving root is always $w\beta$, where $\beta$ is its original root, and the root $w^{-1}\alpha_i$ recorded at a peeling step is the original positive root of the removed summand.

Suppose that a root $\beta$ survived $h$ complete sweeps.  At their boundaries, its current roots would include $\beta,c_Q\beta,\ldots,c_Q^{h-1}\beta$, all positive.  But $c_Q^h=1$, and $1-c_Q$ is invertible on the root space, so
\[
 \sum_{k=0}^{h-1}c_Q^k\beta=0.
\]
A sum of positive roots cannot vanish.  Hence every summand is removed within $h$ sweeps, proving termination and identifying the output.
\end{proof}

\begin{example}
In type~$A_3$ with $\lambda=\omega_2$, take $Q:1\leftarrow2\to3$ and $\delta:=\lambda-w_0\lambda=\alpha_1+2\alpha_2+\alpha_3$, and use the source order $2,1,3$.  Reflection at $2$ gives $d=\alpha_1+\alpha_3$.  The next two steps peel $\alpha_1$ and $\alpha_3$, recording
\[
 s_2\alpha_1=\alpha_1+\alpha_2,\qquad
 s_2s_1\alpha_3=\alpha_2+\alpha_3.
\]
Thus $R_Q(\delta)=\{\alpha_1+\alpha_2,\alpha_2+\alpha_3\}$.  We use this decomposition in Example~\ref{ex:A3}.
\end{example}

The algorithm also characterizes the decomposition in terms of the Euler form.

\begin{proposition}\label{prop:euler-characterization}
Let $\mu,\nu\in\Omega$ with $\mu\geq\nu$, and set $\delta:=\mu-\nu$.  Then $R_Q(\delta)$ is the unique set $R$ of positive roots such that
\[
 \sum_{\beta\in R}\beta=\delta,\qquad
 E_Q(\beta,\gamma)=0\quad(\beta,\gamma\in R,\ \beta\ne\gamma).
\]
\end{proposition}
\begin{proof}
The set $R_Q(\delta)$ satisfies these conditions by Proposition~\ref{prop:generic}.  To prove uniqueness, follow the algorithm for any set $R$ satisfying them, deleting $\alpha_i$ at a peeling step and applying $s_i$ to every root at a reflection step.  At a source $i$ of the current orientation, write $E$ for its Euler form.  Suppose first that $\alpha_i\in R$.  For any other root $\beta=\sum_jb_j\alpha_j$ in the current set, the two Euler equations are
\[
 0=E(\beta,\alpha_i)=b_i,\qquad
 0=E(\alpha_i,\beta)=b_i-\sum_{j\sim i}b_j.
\]
Thus $b_i=0$ and every neighboring coefficient is zero, giving the peeling criterion.  Conversely, when the peeling criterion holds, $s_i d$ has a negative coordinate.  Since $s_i$ sends every positive root other than $\alpha_i$ to a positive root, $\alpha_i$ must occur.  Every set satisfying the displayed conditions therefore follows the same steps.  Equation~\eqref{eq:euler-reflection} preserves the Euler pairings at reflection steps, and the other roots are fixed at peeling steps.  Termination reconstructs the same roots backwards.
\end{proof}

\section{The normalized Coxeter action}\label{sec:coxeter}

After a full source sweep, the orientation is again $Q$: each edge has been reversed once at each endpoint.  Composing the signed identities in Proposition~\ref{prop:transport} gives $\dot c_Qv^Q_\pi=\pm v^Q_{\gamma_{c_Q}(\pi)}$, proving the first assertion of Theorem~\ref{thm:intro-coxeter}.  More generally, source and sink reflections connect all orientations of the Dynkin tree, so a product of their Tits representatives carries any quiver basis to any other, up to signs.

It remains to rescale the vectors so that the scalar factor is independent of $\pi$.  Write $c:=c_Q$, $\dot c:=\dot c_Q$, and $\zeta:=\exp(2\pi i/h)$.  We will choose $\widehat v^Q_\pi$ so that
\[
 T_Q:=\zeta^{|P|}\dot c
\]
permutes them exactly as $\gamma_c$ permutes $\RPP_2(P)$.  This changes the representatives of the basis lines; the integral vectors $v^Q_\pi$ retain their endpoint normalization.  The operator $T_Q$ includes the displayed scalar and need not itself be a Tits representative.

The normalization comes from the ambient minuscule representation.  We first change the signs of its weight vectors so that the opposite-vertex products in each quiver cube have equal coefficients.  We then change their phases so that the Coxeter representative acts with the same scalar on every weight vector.  Taking symmetric squares gives the required representatives of the quiver basis lines.

\subsection{The Euler form and the weight basis}
For this section, relabel a source order for $Q$ as $1,\ldots,r$, so that $c=s_r\cdots s_1$ and $\dot c=n_r\cdots n_1$.  Write $E:=E_Q$, extended bilinearly to the real root space.  Its matrix in simple-root coordinates is upper triangular.  On a complex vector space with basis $\{u^Q_\mu:\mu\in\Omega\}$, define
\begin{equation}\label{eq:cocycle}
 X_\alpha u^Q_\mu:=
 \begin{cases}
 (-1)^{E(\alpha,\mu-\lambda)}u^Q_{\mu+\alpha}&\mu+\alpha\in\Omega,\\
 0&\text{otherwise},
 \end{cases}
 \qquad \alpha\in\Phi.
\end{equation}
The exponent is an integer because $\mu-\lambda$ lies in the root lattice.

\begin{lemma}\label{lem:cocycle}
The operators $e_i=X_{\alpha_i}$ and $f_i=-X_{-\alpha_i}$, together with the weight operators, identify this space with $V^\lambda$.  Under the identification taking $u^Q_\lambda$ to $u_\lambda$, each $u^Q_\mu$ is $u_\mu$ or $-u_\mu$.  Every $X_\alpha$ is a nonzero scalar multiple of a root operator.
\end{lemma}
\begin{proof}
Set $\kappa(x,y):=(-1)^{E(x,y)}$ for $x,y\in L$.  Equation~\eqref{eq:euler} gives
\[
 \kappa(x,y)\kappa(y,x)=(-1)^{(x,y)},
 \qquad \kappa(\alpha,\alpha)=-1\quad(\alpha\in\Phi).
\]
Suppose that $\alpha+\beta\in\Phi$.  First assume that $\mu+\alpha+\beta\in\Omega$.  The minuscule pairings imply that exactly one of $\mu+\alpha$ and $\mu+\beta$ lies in $\Omega$.  If $\mu+\beta\in\Omega$, the coefficient of $u^Q_{\mu+\alpha+\beta}$ in $X_\alpha X_\beta u^Q_\mu$ is
\[
 \kappa(\beta,\mu-\lambda)\kappa(\alpha,\mu+\beta-\lambda)
 =\kappa(\alpha+\beta,\mu-\lambda)\kappa(\alpha,\beta).
\]
If $\mu+\alpha\in\Omega$, the coefficient in $X_\beta X_\alpha u^Q_\mu$ instead has the factor $\kappa(\beta,\alpha)=-\kappa(\alpha,\beta)$.  If $\mu+\alpha+\beta\notin\Omega$, both compositions vanish.  Hence
\[
 [X_\alpha,X_\beta]=\kappa(\alpha,\beta)X_{\alpha+\beta}.
\]
For orthogonal roots, the same calculation gives $[X_\alpha,X_\beta]=0$.  For distinct nonopposite roots of positive inner product, both compositions vanish.  Finally, calculation on a two-element root string gives $[X_\alpha,X_{-\alpha}]=-h_\alpha$, where $h_\alpha$ acts on weight $\mu$ by $(\alpha,\mu)$.  These identities give the Chevalley and Serre relations with the stated choice $f_i=-X_{-\alpha_i}$.

The vector $u^Q_\lambda$ is a highest-weight vector of weight $\lambda$, and the connected weight graph generates the whole space.  Thus the representation is $V^\lambda$.  All nonzero simple lowering coefficients are signs, so the identification taking $u^Q_\lambda$ to $u_\lambda$ gives $u^Q_\mu=\pm u_\mu$.  Iterated brackets of the simple-root operators prove the assertion for every root $\alpha$.
\end{proof}

We henceforth regard $u^Q_\mu$ as a vector in $V^\lambda$ under this identification.  This choice of signs makes the opposite-vertex products in every quiver cube have equal coefficients.  To see this, fix a comparable pair $\mu\geq\nu$ in $\Omega$, set $R:=R_Q(\mu-\nu)$, and use $F_\beta=X_{-\beta}$ in the cube construction, with $u^Q_\mu$ in place of $u_\mu$.  Equation~\eqref{eq:recording} gives $E(\beta,\gamma)=0$ for distinct roots of $R$, so every cross term in the sign exponent vanishes.  Writing $\mu_S=\mu-\sum_{\beta\in S}\beta$, we obtain
\[
 w_S=\left(\prod_{\beta\in S}X_{-\beta}\right)u^Q_\mu
 =(-1)^{-E(\sum_{\beta\in S}\beta,\mu-\lambda)}u^Q_{\mu_S}
\]
and hence
\[
 w_Sw_{R\setminus S}
 =(-1)^{-E(\mu-\nu,\mu-\lambda)}
   u^Q_{\mu_S}u^Q_{\mu_{R\setminus S}}.
\]
The last sign is independent of $S$.  In~\eqref{eq:expansion}, every opposite-vertex monomial therefore has the same coefficient.  For nonempty $R$, complementary subsets repeat each monomial twice; removing this common factor gives an equal-coefficient representative of the quiver basis line.

\subsection{A constant Coxeter phase}
The signs accumulated along a source sweep can be expressed by the Euler form.  This allows one change of phases to make the Coxeter scalar independent of the weight.

\begin{lemma}\label{lem:phase}
There is a vector $y$ in the real root space such that the vectors $\widetilde u^Q_\mu:=e^{\pi i(y,\mu)}u^Q_\mu$ satisfy
\begin{equation}\label{eq:weight-phase}
 \eta\dot c\,\widetilde u^Q_\mu=\widetilde u^Q_{c\mu},
 \qquad
 \eta:=e^{\pi i(\lambda,\lambda)}
 =\exp\!\left(\frac{2\pi i\langle\lambda,\rho^\vee\rangle}{h}\right).
\end{equation}
\end{lemma}
\begin{proof}
The simple-root strings in~\eqref{eq:cocycle} give
\[
 n_i u^Q_\mu
 =(-1)^{E(s_i\mu-\mu,\mu-\lambda)}u^Q_{s_i\mu}.
\]
Set $\mu_0:=\mu$, $\mu_i:=s_i\mu_{i-1}$, and $\Delta_i:=\mu_i-\mu_{i-1}$.  Since $\Delta_i$ is a multiple of $\alpha_i$, upper triangularity gives $E(\Delta_i,\Delta_j)=0$ for $i>j$.  The exponent accumulated in the sweep is consequently
\[
 \sum_i E(\Delta_i,\mu_{i-1}-\lambda)
 =E(c\mu-\mu,\mu-\lambda),
\]
and hence
\[
 \dot c\,u^Q_\mu
 =(-1)^{E(c\mu-\mu,\mu-\lambda)}u^Q_{c\mu}.
\]

Using column coordinates in the ordered simple-root basis, write $c$ and $E$ also for their matrices.  They satisfy
\begin{equation}\label{eq:coxeter-euler}
 c=-E^{-T}E,\qquad(c-1)^TE=-(E+E^T).
\end{equation}
Indeed, the $i$th coordinate after the sweep satisfies
\[
 (cx)_i-\sum_{j<i,\,j\sim i}(cx)_j
 =-x_i+\sum_{j>i,\,j\sim i}x_j,
\]
which is the $i$th row of $E^Tcx=-Ex$.  Since $E+E^T$ is the matrix of the invariant inner product and $(\mu,\mu)=(\lambda,\lambda)$, equation~\eqref{eq:coxeter-euler} gives
\[
 E(c\mu-\mu,\mu-\lambda)
 =-(\mu,\mu-\lambda)=(\lambda,\mu-\lambda).
\]
Choose $y$ satisfying $y-c^{-1}y=-\lambda$, which is possible because $1-c$ is invertible on the root space.  Then
\begin{align*}
 \dot c\,\widetilde u^Q_\mu
 &=e^{\pi i((y,\mu)+(\lambda,\mu-\lambda)-(y,c\mu))}
   \widetilde u^Q_{c\mu}\\
 &=e^{-\pi i(\lambda,\lambda)}\widetilde u^Q_{c\mu}.
\end{align*}
Finally, the root-system identity $\sum_{\alpha>0}(\lambda,\alpha)^2=h(\lambda,\lambda)$ and the dominant minuscule pairings $(\lambda,\alpha)\in\{0,1\}$ give
\[
 h(\lambda,\lambda)=\sum_{\alpha>0}(\lambda,\alpha)
 =2(\lambda,\rho)=2\langle\lambda,\rho^\vee\rangle.
\]
This proves the formula for $\eta$.
\end{proof}

\begin{proof}[Proof of Theorem~\ref{thm:intro-coxeter}]
Fix $y$ as in Lemma~\ref{lem:phase}, and use the corresponding rephased basis $\{\widetilde u^Q_\mu:\mu\in\Omega\}$ and scalar $\eta$.  Let $\pi$ correspond to the comparable pair $\mu\geq\nu$, and set $R:=R_Q(\mu-\nu)$.  The Euler signs give an equal-coefficient cube sum in the $u^Q$-basis.  Replacing each $u^Q$-vector in this sum by its rephased vector preserves the line: every opposite-vertex product has total weight $\xi=\mu+\nu$ and is multiplied by the same phase $e^{\pi i(y,\xi)}$.

Consequently the quiver basis line $\C v^Q_\pi$ contains the vector
\begin{equation}\label{eq:positive-cube}
 \widehat v^Q_{\mu,\nu}
 :=\sum_{\{S,R\setminus S\}}
   \widetilde u^Q_{\mu_S}\widetilde u^Q_{\mu_{R\setminus S}}
 \quad(R\ne\varnothing),\qquad
 \widehat v^Q_{\mu,\mu}:=(\widetilde u^Q_\mu)^2.
\end{equation}
The sum takes each unordered complementary pair of subsets once.  We also write $\widehat v^Q_\pi$ for this vector.

At each step, if the reflected endpoints remain comparable, the simple reflection sends every cube vertex to its reflected vertex.  Otherwise, the current root set contains $\alpha_i$, and $s_i$ exchanges the two vertices of every cube edge corresponding to that root.  It therefore preserves the cube.  Proposition~\ref{prop:transport} then shows that the full sweep $c$ carries the cube for $\pi$ to the cube for $\gamma_c(\pi)$.  By Lemma~\ref{lem:phase}, the operator $\eta^2\dot c$ acts on a product of two rephased weight vectors by applying $c$ to both weight indices.  It consequently sends the sum~\eqref{eq:positive-cube} to the corresponding sum for $\gamma_c(\pi)$.  Since
\[
 \eta^2=\zeta^{\langle2\lambda,\rho^\vee\rangle}=\zeta^{|P|},
\]
we obtain $T_Q\widehat v^Q_\pi=\widehat v^Q_{\gamma_c(\pi)}$, as asserted in~\eqref{eq:intro-coxeter}.
\end{proof}

\section{Cyclic sieving for Coxeter-motion and rowmotion}\label{sec:csp}

Let $P$ be any minuscule poset, and choose a simply laced realization $P=P_\lambda$.  For this realization, use the setup of Section~\ref{sec:background}, fix any orientation $Q$, and use the notation $c=c_Q$, $\dot c=\dot c_Q$, $\zeta$, and $T_Q$ of Section~\ref{sec:coxeter}.  The normalized Coxeter action now counts fixed points by traces.  We identify those traces with evaluations of the size-generating polynomial, and then transfer the result to rowmotion by a conjugacy of toggle products.

\subsection{The character calculation}
The size statistic on reverse plane partitions determines the grading shift in the character.  Set $N:=|P|=\langle2\lambda,\rho^\vee\rangle$.  If $\pi$ has threshold ideals $A\subseteq B$ and corresponding weights $\mu\geq\nu$, then
\[
 |\pi|=2|P|-|A|-|B|.
\]
Applying~\eqref{eq:ideal-weight} to both thresholds yields
\begin{equation}\label{eq:grading}
 |\pi|=N+\langle\mu+\nu,\rho^\vee\rangle
 =N+\langle\xi,\rho^\vee\rangle.
\end{equation}
Thus the size statistic is the principal weight grading, shifted so that its smallest degree is zero.

Let $\chi_{2\lambda}$ denote the character of $V^{2\lambda}$.  Principal specialization replaces each weight monomial $e^\xi$ by $q^{\langle\xi,\rho^\vee\rangle}$.  Equations~\eqref{eq:dimension} and~\eqref{eq:grading}, followed by the Weyl character formula, give
\begin{equation}\label{eq:principal}
 F_{P,2}(q)=q^N\chi_{2\lambda}(q^{\rho^\vee})
 =\prod_{\alpha\in\Phi^+}
 \frac{1-q^{\langle2\lambda+\rho,\alpha^\vee\rangle}}
 {1-q^{\langle\rho,\alpha^\vee\rangle}}.
\end{equation}
The left-hand side shows that the product is a polynomial.  Values at roots of unity are evaluations of this polynomial.

Kostant's Coxeter-element theorem identifies $\dot c$ up to conjugacy with
\[
 s:=\exp(2\pi i\rho^\vee/h).
\]
For the adjoint-group statement, see Prasad~\cite{Prasad}, Proposition~1.  To pass to the simply connected group $G$, observe that all representatives of a fixed Coxeter element are conjugate by the maximal torus.  Indeed, the map $t\mapsto t\,c(t)^{-1}$ is surjective because $1-c$ is invertible on the real character space of the torus.  Central multiples of a representative are therefore conjugate as well, removing the ambiguity in lifting the adjoint-group conjugacy.

For every integer $k$, the permutation action of $T_Q=\zeta^N\dot c$ on the basis $\{\widehat v^Q_\pi\}$ now gives
\begin{align*}
 \#\Fix(\gamma_c^k)
 &=\tr\!\left(T_Q^k\mid V^{2\lambda}\right)\\
 &=\zeta^{kN}\chi_{2\lambda}\!\left(\exp(2\pi ik\rho^\vee/h)\right)\\
 &=F_{P,2}(\zeta^k).
\end{align*}
The scalar $\zeta^N$ in the normalized action supplies exactly the shift in~\eqref{eq:grading}.

The same calculation proves periodicity.  The operator $T_Q$ is conjugate to $\zeta^Ns$, whose eigenvalue on the weight space of weight $\xi$ is
\[
 \exp\!\left(\frac{2\pi i}{h}
 \bigl(N+\langle\xi,\rho^\vee\rangle\bigr)\right).
\]
The expression in parentheses is an integer by~\eqref{eq:grading} and~\eqref{eq:dimension}.  Hence $T_Q^h=1$, and its permutation of the basis gives $\gamma_c^h=1$.

\subsection{From Coxeter-motion to rowmotion}
The conjugacy with rowmotion uses only the involution and commutation relations of the toggles.  The argument therefore applies at every height, without braid relations between adjacent label toggles.

\begin{proposition}\label{prop:row-conjugacy}
On bounded reverse plane partitions of a minuscule heap, piecewise-linear rowmotion is conjugate to Coxeter-motion for every Coxeter element.
\end{proposition}
\begin{proof}
We apply the rank-parity argument of Rush and Shi~\cite{RS}, Section~6.2, in the piecewise-linear toggle group; see also Striker and Williams~\cite{SW} and Hopkins~\cite{Hopkins}, Remark~4.22.

First consider involutions indexed by the vertices of a tree, with commuting involutions at nonadjacent vertices.  All products containing each involution once are conjugate.  To see this, orient each edge from the vertex whose involution acts earlier to the one whose involution acts later.  Orders defining the same orientation differ by swaps of consecutive nonadjacent vertices, so they give the same product.  Conjugating by the first or last factor moves it to the other end, reversing the arrows at that source or sink.  These reversals connect all orientations: remove a leaf and argue by induction, reversing the leaf when necessary before reversing its neighbor, and adjusting the leaf edge at the end.

The minuscule heap is graded of rank $h-2$, with rank function $\operatorname{rk}$ normalized to be zero at its minimum.  Let $r_j$ be the product of all element toggles in rank $j$.  Toggles in one rank commute, so $r_j$ is an involution, and $r_j$ commutes with $r_k$ when $|j-k|>1$.  Applying the preceding fact to the path of ranks shows that rowmotion, which applies ranks in decreasing order, is conjugate to the product that applies all even ranks and then all odd ranks.

Color the vertices of the Dynkin diagram with two colors so that adjacent vertices have different colors.  Every cover in the connected minuscule heap joins adjacent Dynkin labels, so the elements in even ranks have labels of one color and those in odd ranks have labels of the other.  The even-then-odd product is consequently Coxeter-motion for the Coxeter element that applies the simple reflections of one color and then those of the other.  Finally, label toggles are involutions and commute at nonadjacent Dynkin vertices.  Applying the same tree argument to the Dynkin diagram conjugates this product to the product for any Coxeter element.  Combining the two conjugacies proves the proposition.
\end{proof}

Returning to height two, conjugate permutations have the same fixed-point counts for every power.  Thus
\[
 \#\Fix(\row^k)=\#\Fix(\gamma_c^k)=F_{P,2}(\zeta^k)
 \qquad(k\in\Z).
\]
To show that their order is exactly $h$, it remains to exhibit an orbit of length $h$.  As observed by Hopkins~\cite{Hopkins}, Remark~4.22, the $h$ distinct functions
\[
 \pi_j(p):=
 \begin{cases}
 0&\operatorname{rk}(p)<j,\\
 2&\operatorname{rk}(p)\geq j,
 \end{cases}
 \qquad 0\leq j<h,
\]
form one rowmotion orbit: the maximal-first toggle order sends $\pi_j$ to $\pi_{j+1}$, with subscripts taken modulo $h$.  This completes the proof of Corollary~\ref{cor:intro-csp}.

\section{Comparison with canonical bases}\label{sec:comparison}

For the standard orientation in type~$A$, the quiver basis agrees with the canonical basis up to rescaling.  A noncrossing matching determines the generic root decomposition, and Lam's Pl\"ucker-product expansion identifies the corresponding cube vector by duality.  We use the lower global basis of $V^{2\lambda}$ specialized at $q=1$, and the basis dual to it in $(V^{2\lambda})^*$.  

\subsection{Equioriented type~\texorpdfstring{$A$}{A}}
Fix $n\geq2$ and $1\leq k\leq n-1$.  For the type~$A$ discussion, take $\mathfrak g=\mathfrak{sl}_n$, $G=\operatorname{SL}_n(\C)$, and $\lambda=\omega_k$, with the standard matrix Chevalley generators.  Thus $V^\lambda=\bigwedge^k\C^n$ and $P=[k]\times[n-k]$.  First let $Q:1\to2\to\cdots\to n-1$.  
Our product convention gives $c_Q=(1\,n\,n-1\,\cdots\,2)$, the inverse of the cycle in the subset example.  An operator permutes a basis up to scalars if and only if its inverse does.  
Write $\mathbf e_1,\ldots,\mathbf e_n$ for the standard basis of $\C^n$ and $\varepsilon_1,\ldots,\varepsilon_n$ for the coordinate vectors of $\R^n$.  The simple roots are $\alpha_i=\varepsilon_i-\varepsilon_{i+1}$.  
For $A=\{a_1<\cdots<a_k\}\subseteq[n]$, write
\[
 u_A:=\mathbf e_{a_1}\wedge\cdots\wedge\mathbf e_{a_k},\qquad
 \mu(A):=\wt(u_A)=\sum_{j\in A}\varepsilon_j-\frac{k}{n}\sum_{j=1}^n\varepsilon_j.  
\]
For $B=\{b_1<\cdots<b_k\}\subseteq[n]$, the root order becomes the componentwise order on subsets, with its direction reversed:
\begin{equation}\label{eq:lam-comparable}
 \mu(A)\geq\mu(B)
 \quad\Longleftrightarrow\quad
 a_j\leq b_j\quad(1\leq j\leq k)
 \quad\Longleftrightarrow\quad
 |A\cap[t]|\geq|B\cap[t]|\quad(1\leq t\leq n).  
\end{equation}
Thus $\mu(A)\geq\mu(B)$ exactly when $A$ and $B$ are the columns of a semistandard tableau $T$ of shape $(2^k)$.  Write $\pi:=\pi_T$ for the corresponding element of $\RPP_2(P)$, with endpoint weights $\mu=\mu(A)$ and $\nu=\mu(B)$.  

Place the numbers $1,\ldots,n$ on a line and mark the common entries
$U:=A\cap B$.  On the remaining vertices, write an opening parenthesis
at each element of $A\setminus B$ and a closing parenthesis at each
element of $B\setminus A$.  Match each closing parenthesis to the most
recent unmatched opening parenthesis.  The prefix inequalities
in~\eqref{eq:lam-comparable} ensure that every step is possible.  
Write $M:=M(T)$ for the resulting matching.  Then $(M,U)$ is a partial noncrossing matching with
\[
 |M|+|U|=k,
\]
where $|M|$ counts arcs.  Conversely, placing the smaller endpoint of
each arc in $A$, the larger endpoint in $B$, and every marked vertex
in both columns recovers the tableau.  This is the tableau-to-matching
correspondence of Lam~\cite{Lam}, Proposition~4.  Throughout this section, $(p,q)$, with $p<q$, denotes the unordered edge $\{p,q\}$.

The arcs also recover the generic quiver representation.  
Associate to $(p,q)\in M$, with $p<q$, the positive root
\[
 \beta_{pq}:=\eps_p-\eps_q=\alpha_p+\cdots+\alpha_{q-1}
\]
and its interval representation $M_{pq}$, supported on
$p,p+1,\ldots,q-1$.  These roots sum to $\mu(A)-\mu(B)$.  
For distinct arcs, the intervals are disjoint or strictly nested.  
The formulas
\[
 \Hom(M_{pq},M_{rs})\ne0
 \quad\Longleftrightarrow\quad r\leq p<s\leq q,
 \qquad
 E_Q(\beta_{pq},\beta_{rs})
 =\mathbf1_{[r,s)}(p)-\mathbf1_{[r,s)}(q)
\]
then show that both Hom spaces and both Euler pairings vanish.  
The Ext spaces vanish by~\eqref{eq:hom-ext}, so the direct sum of the
intervals is rigid.  Consequently,
\begin{equation}\label{eq:lam-recording}
 R_Q(\mu(A)-\mu(B))=\{\beta_{pq}:(p,q)\in M\}.  
\end{equation}
The parenthesis matching therefore gives the generic decomposition directly.  

Each cube vertex chooses one endpoint of every arc.  More explicitly,
for $S\subseteq M$, set
\begin{equation}\label{eq:lam-cube-subsets}
 A_S:=
 U\cup\{q:(p,q)\in S\}\cup\{p:(p,q)\in M\setminus S\}.  
\end{equation}
Then
\[
 \mu(A_S)=\mu(A)-\sum_{(p,q)\in S}\beta_{pq},
 \qquad
 A_\varnothing=A,\quad A_M=B.  
\]
Opposite vertices are $A_S$ and $A_{M\setminus S}$: they contain
the marked vertices in common and choose opposite endpoints on every
arc.  In the cube expansion, take $F_{\beta_{pq}}=E_{qp}$, where $E_{qp}$ is the matrix unit sending $\mathbf e_p$ to $\mathbf e_q$ and annihilating the other coordinate vectors.  
Reversing the endpoint choice on an arc $(p,q)$ changes the product
of the two wedge signs by the parity of the entries strictly between
$p$ and $q$ in both factors.  An enclosed arc contributes two entries,
and an enclosed marked vertex also contributes two.  Noncrossing
ensures that there are no other contributions.  All opposite-pair
coefficients therefore agree.  With the endpoint normalization,
\begin{equation}\label{eq:lam-quiver-vector}
 v^Q_\pi
 =\sum_{\{S,M\setminus S\}}u_{A_S}u_{A_{M\setminus S}}
 \qquad(M\ne\varnothing),
\end{equation}
where the sum contains one term for each unordered complementary pair.  
For $M=\varnothing$ the vector is $u_A^2$.  

We identify the line of $v^Q_\pi$ by dualizing Lam's expansion of Pl\"ucker products.  Let $R(k,n)$ denote the homogeneous
coordinate ring of the Grassmannian $\operatorname{Gr}(k,n)$ in its Pl\"ucker embedding, and write $\Delta_I$
for its Pl\"ucker coordinate indexed by $I$.  
Its degree-two component is $(V^{2\omega_k})^*$.  
Lam~\cite{Lam}, Theorem~1(ii) and Proposition~4, identifies the
invariant $\Delta_{M,U}$ with the dual canonical vector $H(T)$
corresponding to the tableau above.  
For arbitrary $k$-subsets $I,J$, his Theorem~3 gives
\begin{equation}\label{eq:lam-product}
 \Delta_I\Delta_J
 =\sum_{(N,U')\in\mathcal C(I,J)}\Delta_{N,U'},
\end{equation}
where $\mathcal C(I,J)$ consists of partial noncrossing matchings
on $I\mathbin\triangle J$, with marked set $U'=I\cap J$
and one endpoint of each arc in each
of $I\setminus J$ and $J\setminus I$.  
For a fixed matching $(M,U)$, this compatibility condition says exactly
that $\{I,J\}$ is an opposite pair of its cube.  

Let $G(T)$ be the canonical vector dual to $H(T)$.  
Pairing~\eqref{eq:lam-product} with $G(T)$ gives
\begin{equation}\label{eq:lam-indicator}
 \langle\Delta_I\Delta_J,G(T)\rangle
 =\begin{cases}
 1&(M,U)\in\mathcal C(I,J),\\
 0&\text{otherwise}.  
 \end{cases}
\end{equation}
Thus Lam's product expansion records precisely the support of the
quiver vector.  

\begin{proposition}\label{prop:type-a}
For the equioriented type~$A$ quiver, the quiver basis agrees with the
canonical basis at $q=1$ up to rescaling.  Under the tableau-to-matching
correspondence, its dual basis agrees with Lam's invariants
$\Delta_{M,U}$ up to rescaling.  
\end{proposition}
\begin{proof}
Use the Cartan embedding into $\Sym^2(\bigwedge^k\C^n)$ that sends
the highest vector to $u_{[k]}^2$.  Multiplication of Pl\"ucker
coordinates is dual to this embedding, for the symmetric-square pairing
\[
 \langle\Delta_I\Delta_J,u_Ku_L\rangle
 =\tfrac12(\delta_{IK}\delta_{JL}+\delta_{IL}\delta_{JK}).  
\]
Equation~\eqref{eq:lam-indicator} therefore determines the ambient
monomial coefficients of $G(T)$, and gives the explicit formula
\begin{equation}\label{eq:lam-canonical-vector}
 G(T)
 =\sum_{S\subseteq M}u_{A_S}u_{A_{M\setminus S}}
 =\begin{cases}
 2v^Q_\pi&\mu\ne\nu,\\
 v^Q_\pi&\mu=\nu.  
 \end{cases}
\end{equation}
Comparing with~\eqref{eq:lam-quiver-vector} proves the assertion.  
The factor two comes from the symmetric-square pairing.  
\end{proof}

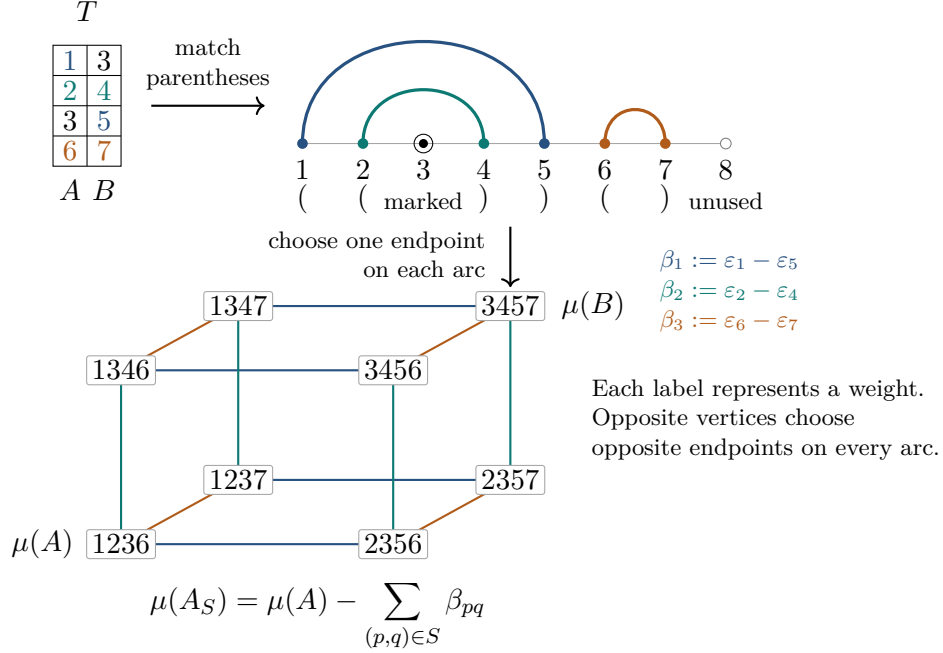
\begin{figure}[tbp]
\centering
\begingroup
\definecolor{lamouter}{RGB}{39,81,128}
\definecolor{laminner}{RGB}{12,121,115}
\definecolor{lamright}{RGB}{176,91,25}
\begin{tikzpicture}[x=1cm,y=1cm,font=\small,
  every node/.style={inner sep=2pt}]
 \node at (0.65,2.05) {$T$};
 \draw[thin] (0.2,0) rectangle (1.1,1.6);
 \draw[thin] (0.65,0)--(0.65,1.6);
 \foreach \y in {0.4,0.8,1.2}
   \draw[thin] (0.2,\y)--(1.1,\y);
 \node[lamouter] at (0.425,1.4) {$1$};
 \node[laminner] at (0.425,1.0) {$2$};
 \node at (0.425,0.6) {$3$};
 \node[lamright] at (0.425,0.2) {$6$};
 \node at (0.875,1.4) {$3$};
 \node[laminner] at (0.875,1.0) {$4$};
 \node[lamouter] at (0.875,0.6) {$5$};
 \node[lamright] at (0.875,0.2) {$7$};
 \node at (0.425,-0.32) {$A$};
 \node at (0.875,-0.32) {$B$};
 \draw[->,thick] (1.5,0.8)--(3.0,0.8);
 \node[align=center,font=\scriptsize] at (2.25,1.35)
   {match\\parentheses};

 \draw[black!35] (3.5,0.3)--(9.1,0.3);
 \draw[lamouter,very thick] (3.5,0.3)
   .. controls (3.5,2.1) and (6.7,2.1) .. (6.7,0.3);
 \draw[laminner,very thick] (4.3,0.3)
   .. controls (4.3,1.25) and (5.9,1.25) .. (5.9,0.3);
 \draw[lamright,very thick] (7.5,0.3)
   .. controls (7.5,0.9) and (8.3,0.9) .. (8.3,0.3);
 \foreach \x/\lab in {3.5/1,4.3/2,5.1/3,5.9/4,6.7/5,7.5/6,8.3/7,9.1/8}
   \node at (\x,-0.05) {$\lab$};
 \foreach \x in {3.5,6.7} \fill[lamouter] (\x,0.3) circle (2pt);
 \foreach \x in {4.3,5.9} \fill[laminner] (\x,0.3) circle (2pt);
 \foreach \x in {7.5,8.3} \fill[lamright] (\x,0.3) circle (2pt);
 \fill (5.1,0.3) circle (1.7pt);
 \draw (5.1,0.3) circle (3.5pt);
 \draw[black!50,fill=white] (9.1,0.3) circle (2pt);
 \foreach \x in {3.5,4.3,7.5} \node at (\x,-0.43) {$($};
 \foreach \x in {5.9,6.7,8.3} \node at (\x,-0.43) {$)$};
 \node[font=\scriptsize] at (5.1,-0.46) {marked};
 \node[font=\scriptsize] at (9.1,-0.46) {unused};

 \node[lamouter,anchor=west,font=\scriptsize] at (8.15,-1.25)
   {$\beta_1:=\varepsilon_1-\varepsilon_5$};
 \node[laminner,anchor=west,font=\scriptsize] at (8.15,-1.65)
   {$\beta_2:=\varepsilon_2-\varepsilon_4$};
 \node[lamright,anchor=west,font=\scriptsize] at (8.15,-2.05)
   {$\beta_3:=\varepsilon_6-\varepsilon_7$};
 \draw[->,thick] (6.25,-0.8)--(6.25,-1.6);
 \node[anchor=east,align=right,font=\scriptsize] at (6.0,-1.16)
   {choose one endpoint\\on each arc};

 \coordinate (c000) at (1.1,-5.0);
 \coordinate (c100) at (4.7,-5.0);
 \coordinate (c010) at (1.1,-2.7);
 \coordinate (c110) at (4.7,-2.7);
 \coordinate (c001) at (2.65,-4.15);
 \coordinate (c101) at (6.25,-4.15);
 \coordinate (c011) at (2.65,-1.85);
 \coordinate (c111) at (6.25,-1.85);
 \foreach \a/\b in {000/100,010/110,001/101,011/111}
   \draw[lamouter,thick] (c\a)--(c\b);
 \foreach \a/\b in {000/010,100/110,001/011,101/111}
   \draw[laminner,thick] (c\a)--(c\b);
 \foreach \a/\b in {000/001,100/101,010/011,110/111}
   \draw[lamright,thick] (c\a)--(c\b);
 \foreach \bits/\lab in {000/1236,100/2356,010/1346,110/3456,
                         001/1237,101/2357,011/1347,111/3457}
   \node[fill=white,draw=black!35,rounded corners=1pt] at (c\bits)
     {$\lab$};
 \node[anchor=east] at (0.55,-5.0) {$\mu(A)$};
 \node[anchor=west] at (6.85,-1.85) {$\mu(B)$};
 \node[align=left,anchor=west,font=\scriptsize] at (7.25,-3.35)
   {Each label represents a weight.\\
    Opposite vertices choose\\opposite endpoints on every arc.};
 \node at (3.7,-5.95)
   {$\displaystyle
     \mu(A_S)=\mu(A)-\sum_{(p,q)\in S}\beta_{pq}$};
\end{tikzpicture}
\caption{From the tableau $T$ with columns $A=1236$ and $B=3457$
to Lam's partial noncrossing matching and the quiver cube,
for $V^\lambda=\bigwedge^4\mathbb C^8$.  The marked entry $3$ belongs to every cube vertex.  
An arc and the parallel cube edges in its color correspond to the same
root.  A vertex labeled $I$ represents the weight $\mu(I)$.  
The four opposite vertex pairs give the four terms
of $v^Q_\pi$.}
\label{fig:lam-matching}
\endgroup
\end{figure}

\begin{example}\label{ex:lam-matching}
The matching that determines a quiver vector need not be the only matching in the associated Pl\"ucker product.  Take $k=4$, $n=8$, $A=1236$, and $B=3457$, and let $T$ be the tableau with columns $A,B$, so that $\pi=\pi_T$ has endpoint weights $\mu(A),\mu(B)$.  
The marked set is $U=\{3\}$, and parenthesis matching gives
\[
 M=\{(1,5),(2,4),(6,7)\},\qquad
 R_Q(\mu(A)-\mu(B))
 =\{\eps_1-\eps_5,\eps_2-\eps_4,\eps_6-\eps_7\}.  
\]
Figure~\ref{fig:lam-matching} displays the tableau, matching, and cube.  
The four opposite vertex pairs give
\begin{align*}
 v^Q_\pi={}&u_{1236}u_{3457}+u_{1237}u_{3456}\\
 &{}+u_{1346}u_{2357}+u_{1347}u_{2356}.  
\end{align*}
On the dual side, the associated invariant is
\[
 \Delta_{M,\{3\}}
 =\Delta_{1236}\Delta_{3457}
 -\Delta_{1235}\Delta_{3467}
 +\Delta_{1234}\Delta_{3567}.  
\]
Indeed, set
$N:=\{(1,7),(2,4),(5,6)\}$ and
$L:=\{(1,7),(2,6),(4,5)\}$.  
The matchings compatible with $(1236,3457)$ are $M$ and $N$;
those compatible with $(1235,3467)$ are $N$ and $L$;
and only $L$ is compatible with $(1234,3567)$.  
The displayed identity follows by subtracting their instances of
\eqref{eq:lam-product}.  
Although $N$ is compatible with $(A,B)$, its arc $(5,6)$ has its
smaller endpoint in $B$.  Thus $M$ is the unique matching whose
left endpoints are the entries of $A\setminus B$; the other compatible
matching occurs in the Pl\"ucker product expansion, not in the generic decomposition
of $\mu(A)-\mu(B)$.  
\end{example}

\subsection{Changing the orientation}

Changing the orientation can require more than rescaling the basis vectors.  
For the two orientations of $A_3$ considered below, the cube with endpoints
$12$ and $34$ has different intermediate vertices and gives different
vectors.  The resulting change of basis is needed to lift the new
Coxeter element: the image of one canonical basis vector has two nonzero canonical coordinates, so this element does not permute the canonical basis up to scalars.  

\begin{example}\label{ex:A3}
In $\Sym^2(\bigwedge^2\C^4)$, set
\[
 M_1:=u_{12}u_{34},\qquad M_2:=u_{13}u_{24},\qquad
 M_3:=u_{14}u_{23}.  
\]
The subspace of $V^{2\omega_2}$ with $\operatorname{GL}_4$ weight $(1,1,1,1)$, or $\mathfrak{sl}_4$ weight zero, consists of
$xM_1+yM_2+zM_3$ with $x-y+z=0$.  
Its two canonical basis lines have representatives
\[
 G_1:=M_1+M_2,\qquad G_2:=M_2+M_3.
\]
By Proposition~\ref{prop:type-a}, these are also the endpoint-normalized
quiver vectors for $Q:1\to2\to3$.  
For $Q':1\leftarrow2\to3$, the generic decompositions for the comparable
pairs $(12,34)$ and $(13,24)$ give the endpoint-normalized vectors
\begin{equation}\label{eq:A3-cubes}
 v_1:=M_1-M_3,\qquad v_2:=M_2+M_3.
\end{equation}
Indeed, the first decomposition consists of $\eps_1-\eps_3$ and
$\eps_2-\eps_4$, the second of $\eps_1-\eps_2$ and
$\eps_3-\eps_4$, and
\[
 E_{31}E_{42}(u_{12}^2)=2(M_1-M_3),\qquad
 E_{21}E_{43}(u_{13}^2)=2(M_2+M_3).  
\]
Figure~\ref{fig:A3-cubes} displays the cube for $(12,34)$ in both
orientations.  In the figure we abbreviate
$v^Q_{A,B}:=v^Q_{\mu(A),\mu(B)}$.  

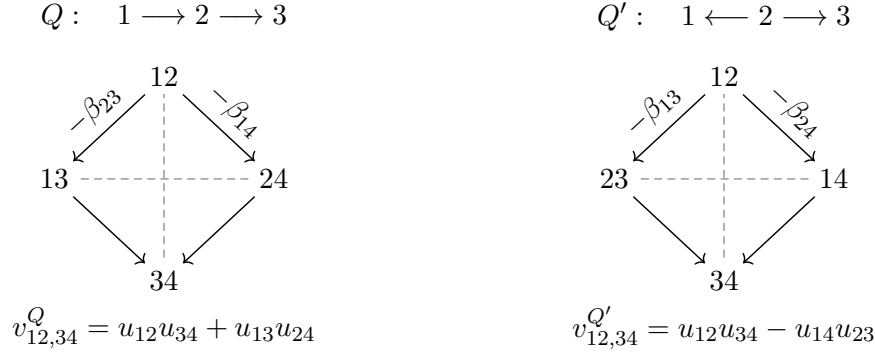
\begin{figure}[tbp]
\centering
\begin{tikzpicture}[
  x=1cm,y=1cm,
  every node/.style={font=\small},
  vertex/.style={fill=white,inner sep=3pt},
  root edge/.style={->,line width=.55pt},
  opposite/.style={densely dashed,black!45,line width=.5pt}]
  \begin{scope}
    \node at (0,.1) {$Q:\quad 1\longrightarrow2\longrightarrow3$};
    \coordinate (t) at (0,-.7);
    \coordinate (l) at (-1.45,-2.05);
    \coordinate (r) at (1.45,-2.05);
    \coordinate (b) at (0,-3.4);
    \draw[opposite] (t)--(b) (l)--(r);
    \node[vertex] (T) at (t) {$12$};
    \node[vertex] (L) at (l) {$13$};
    \node[vertex] (R) at (r) {$24$};
    \node[vertex] (B) at (b) {$34$};
    \draw[root edge] (T)--node[midway,above,sloped] {$-\beta_{23}$} (L);
    \draw[root edge] (T)--node[midway,above,sloped] {$-\beta_{14}$} (R);
    \draw[root edge] (L)--(B);
    \draw[root edge] (R)--(B);
    \node at (0,-4.02) {$v^Q_{12,34}=u_{12}u_{34}+u_{13}u_{24}$};
  \end{scope}
  \begin{scope}[xshift=7.4cm]
    \node at (0,.1) {$Q':\quad 1\longleftarrow2\longrightarrow3$};
    \coordinate (t) at (0,-.7);
    \coordinate (l) at (-1.45,-2.05);
    \coordinate (r) at (1.45,-2.05);
    \coordinate (b) at (0,-3.4);
    \draw[opposite] (t)--(b) (l)--(r);
    \node[vertex] (T) at (t) {$12$};
    \node[vertex] (L) at (l) {$23$};
    \node[vertex] (R) at (r) {$14$};
    \node[vertex] (B) at (b) {$34$};
    \draw[root edge] (T)--node[midway,above,sloped] {$-\beta_{13}$} (L);
    \draw[root edge] (T)--node[midway,above,sloped] {$-\beta_{24}$} (R);
    \draw[root edge] (L)--(B);
    \draw[root edge] (R)--(B);
    \node at (0,-4.02) {$v^{Q'}_{12,34}=u_{12}u_{34}-u_{14}u_{23}$};
  \end{scope}
\end{tikzpicture}
\caption{Two orientations, two cubes with the same endpoints in
$V^{\omega_2}$ for type~$A_3$.  
The vertex $ij$ denotes the weight of the increasing wedge $u_{ij}$, and
$\beta_{pq}=\eps_p-\eps_q$.  
Solid arrows subtract roots; parallel edges have the same label.  
Dashed lines join opposite vertices and mark the monomials in the
displayed vectors.}
\label{fig:A3-cubes}
\end{figure}

The Coxeter element for $Q'$ is $c=s_3s_1s_2$.  
Its Tits representative acts on the ambient monomials by
\[
 \dot c M_1=-M_2,\qquad \dot c M_2=-M_1,\qquad
 \dot c M_3=M_3.
\]
It follows that
\[
 \dot c G_1=-G_1,\qquad \dot c G_2=-G_1+G_2,
 \qquad
 \dot c v_1=-v_2,\qquad \dot c v_2=-v_1.
\]
Both canonical coordinates of $\dot c G_2$ are nonzero and remain
nonzero under any rescaling of $G_1,G_2$.  
Thus $\dot c$ does not permute the canonical basis up to scalars,
whereas it interchanges the two quiver vectors up to sign.  
Here $|P|=4$ and $\zeta=i$, so the global Coxeter normalization is
$\zeta^{|P|}=1$.  Taking $v_1,-v_2$ as representatives gives the
unsigned transposition corresponding to Coxeter-motion.  
The source reflection at vertex $1$ changes $Q$ into $Q'$, and its
action on $G_1,G_2$ is visible directly:
\[
 n_1G_1=v_1,\qquad n_1G_2=-v_2.
\]
\end{example}

In fact, we can classify all permutations whose representatives preserve the canonical basis up to scalars.

\begin{proposition}\label{prop:type-a-canonical-symmetries}
Let $2\leq k\leq n-2$, and let $\widetilde w\in\operatorname{SL}_n(\C)$ represent $w\in\mathfrak S_n$ in the normalizer of the diagonal torus.  Then $\widetilde w$ permutes the specialized canonical basis of $V^{2\omega_k}$ up to nonzero scalar multiples if and only if $w$ belongs to the dihedral group generated by $c_0:=(1\,2\,\cdots\,n)$ and $w_0:i\mapsto n+1-i$.  The same criterion holds for the dual canonical basis of $(V^{2\omega_k})^*$.  In particular, the only Coxeter elements with this property are $c_0$ and $c_0^{-1}$.
\end{proposition}

\begin{proof}
Use the natural $\operatorname{GL}_n$-action on $V^{2\omega_k}$.  The representative $\widetilde w$ is a diagonal matrix times the ordinary permutation matrix of $w$.  Diagonal matrices act by nonzero scalars on weight spaces, so it suffices to use ordinary permutation matrices.  Their actions on the products $u_Au_B$ permute these monomials up to signs.

Fix $a<b<c<d$ and a $(k-2)$-subset $C\subseteq[n]\setminus\{a,b,c,d\}$, and set
\[
 \begin{aligned}
 m_1&:=u_{C\cup\{a,b\}}u_{C\cup\{c,d\}},\\
 m_2&:=u_{C\cup\{a,c\}}u_{C\cup\{b,d\}},\\
 m_3&:=u_{C\cup\{a,d\}}u_{C\cup\{b,c\}}.
 \end{aligned}
\]
These monomials have the same weight.  The canonical basis has exactly two vectors of this weight, corresponding to the two noncrossing matchings on $a,b,c,d$ with marked set $C$.  By~\eqref{eq:lam-canonical-vector}, they are proportional to
\[
 m_1+m_2,\qquad m_2+m_3.
\]
The unique monomial occurring in both is $m_2$.  Its two factors record the crossing pair of chords $\{a,c\}$ and $\{b,d\}$ when $1,\ldots,n$ are placed around a circle.

Suppose that $\widetilde w$ permutes the canonical basis up to scalars.  It carries these two vectors to scalar multiples of the two canonical vectors of the image weight.  Since the action permutes the monomials up to nonzero scalars, it carries their unique common monomial to the unique common monomial there.  Consequently, $w$ preserves the crossing pairing of every four-element subset.  Such a $C$ exists for every choice of $a,b,c,d$, since $k\leq n-2$.  The same argument applies to $w^{-1}$.

The cyclically adjacent pairs are exactly those having no crossing partner: for a nonadjacent pair, choosing one vertex from each open arc between its endpoints gives a crossing pair.  Thus $w$ preserves cyclic adjacency, so it belongs to the dihedral group.

Conversely, rotation and reversal preserve noncrossing matchings.  On a monomial $u_Au_B$, the permutation matrix of $w_0$ has sign $1$, since reversing either wedge has sign $(-1)^{\binom{k}{2}}$.  The permutation matrix of $c_0$ has sign
\[
 (-1)^{(k-1)(\mathbf1_{n\in A}+\mathbf1_{n\in B})},
\]
which is constant on each weight space.  Hence~\eqref{eq:lam-canonical-vector} shows that both generators, and therefore the whole dihedral group, permute the canonical basis up to scalars.  The action on the dual basis has the inverse transpose matrix, which permutes that basis up to scalars if and only if the original matrix does.

Finally, a Coxeter element is an $n$-cycle with $n-1$ inversions.  It cannot be a dihedral reflection, whose order is at most two.  The rotation $c_0^j$, for $1\leq j\leq n-1$, has one-line notation $(j+1,\ldots,n,1,\ldots,j)$ and hence $j(n-j)$ inversions.  Equality with $n-1$ forces $(j-1)(n-j-1)=0$, so $j=1$ or $n-1$.
\end{proof}

For $k=1$ or $k=n-1$, every weight space of $V^{2\omega_k}$ is one-dimensional.  Every element of the torus normalizer therefore permutes the canonical basis up to scalars.

In type~$D_4$, the canonical obstruction occurs already for a standard
ordering of the simple reflections.  The vector representation gives
a three-dimensional example in weight zero.  

\begin{example}\label{ex:D4}
Take $V^\lambda=V^{\omega_1}$, the vector representation of type~$D_4$, and write $\eps_1,\ldots,\eps_4$ for the coordinate vectors of $\R^4$.  Choose weight vectors $x_i,y_i$ of weights $\eps_i,-\eps_i$, for $1\leq i\leq4$, normalized so that the invariant quadratic tensor is $\sum_{i=1}^4x_i y_i$.  Their signs are adapted to this invariant form and need not agree with those of the standard weight vectors $u_\mu$.  
Set $z_i:=x_i y_i$.  The zero-weight space of the Cartan square is
\[
 V^{2\omega_1}_0
 =\left\{\sum_{i=1}^4 a_i z_i:\sum_{i=1}^4a_i=0\right\}.  
\]
Its three canonical basis lines have representatives
\[
 G_1:=z_1-z_2,\qquad G_2:=z_2-z_3,\qquad G_3:=z_3-z_4.
\]
To see this, use Bourbaki's labeling
$\alpha_i=\eps_i-\eps_{i+1}$ for $i=1,2,3$ and
$\alpha_4=\eps_3+\eps_4$.  
For $i=1,2,3$, the one-dimensional weight-$\alpha_i$ space is the
highest-weight space of an $i$-string, and lowering gives
$z_{i+1}-z_i$.  String compatibility of the lower global basis,
as in Lusztig~\cite{Lusztig}, therefore supplies these three independent
lines; the displayed dimension formula shows that they exhaust the
canonical basis lines of weight zero.  

For $c=s_1s_2s_3s_4$, the Tits representative sends
$z_1\mapsto z_2\mapsto z_3\mapsto z_1$ and fixes $z_4$.  
Consequently,
\[
 \dot c G_1=G_2,\qquad
 \dot c G_2=-G_1-G_2,\qquad
 \dot c G_3=G_1+G_2+G_3.
\]
The second and third images each have more than one nonzero canonical
coordinate, so $\dot c$ does not permute the canonical basis up to scalars.  
The corresponding orientation is $Q:3\to2\leftarrow4$, $2\to1$.  For the three comparable pairs $(\eps_i,-\eps_i)$, $1\leq i\leq3$, the generic decompositions are
\[
 R_Q(2\eps_i)=\{\eps_i-\eps_4,\eps_i+\eps_4\}.  
\]
The two roots in each pair have Euler matrix equal to the identity, which identifies the generic decomposition by Proposition~\ref{prop:euler-characterization}.  Their cube has opposite products $z_i$ and $z_4$, so its vector has the form $z_i+a z_4$ after rescaling.  The zero-weight equation for the Cartan square gives $a=-1$.  Thus the quiver basis vectors are proportional to
\[
 v_1:=z_1-z_4,\qquad v_2:=z_2-z_4,\qquad v_3:=z_3-z_4,
\]
and $\dot c$ sends $v_1\mapsto v_2\mapsto v_3\mapsto v_1$.  
Strayer~\cite{Strayer}, Section~3 and Theorem~9.5, constructs his basis by applying divided powers of the simple lowering operators along a linear extension of the minuscule heap.  Here such a linear extension has labels $1,2,3,4,2,1$, and his three vectors of weight zero are scalar multiples of $G_1,G_2,G_3$.  Thus, in this example, $\dot c$ does not permute Strayer's basis up to scalars.
\end{example}

\section{The canonical obstruction outside type \texorpdfstring{$A$}{A}}
\label{sec:canonical-obstruction}
\label{subsec:canonical-obstruction}

Changing the Coxeter element does not restore the compatibility that fails in Example~\ref{ex:D4}.  The following theorem applies to every minuscule Cartan square of type~$D$ or~$E$.

\begin{theorem}\label{thm:canonical-obstruction}
Let $\lambda$ be a nonzero dominant minuscule weight in type $D$ or $E$.  No representative of a Coxeter element in the torus normalizer permutes the specialized canonical basis of $V^{2\lambda}$ up to nonzero scalar multiples.  The same statement holds for the dual canonical basis of $(V^{2\lambda})^*$.
\end{theorem}

The proof uses the nonnegative coefficients of canonical vectors in the products $u_\mu u_\nu$.  A Tits representative permutes these products up to sign.  If it also permutes the canonical basis up to scalars, the nonzero coefficients of the image of each canonical vector must all have the same sign.

\begin{lemma}\label{lem:canonical-sign}
Assume that $\Phi$ is simply laced.  Use the standard weight basis $\{u_\mu\}$ of $V^\lambda$, in which all nonzero simple raising and lowering coefficients are $1$, and embed $V^{2\lambda}$ in $\Sym^2 V^\lambda$ by mapping its highest-weight canonical vector to $u_\lambda^2$.  Every specialized lower-canonical vector has nonnegative coefficients in the monomials $u_\mu u_\nu$.

Suppose a Tits representative acts by
\[
 \dot w u_\mu=\epsilon(\mu)u_{w\mu},
 \qquad \epsilon(\mu)\in\{1,-1\},
\]
and set $D_\epsilon u_\mu:=\epsilon(\mu)u_\mu$.  If $\dot w$ permutes the canonical basis of $V^{2\lambda}$ up to scalars, then $\Sym^2(D_\epsilon)$ preserves the Cartan square.  Moreover, for every canonical vector $b$, the nonzero coefficients of $\dot w b$ in the products $u_\mu u_\nu$ all have the same sign.
\end{lemma}

\begin{proof}
Let
\[
 \Gamma:V^{2\lambda}\longrightarrow V^\lambda\otimes V^\lambda
\]
be the $G$-module embedding taking the highest-weight canonical vector to $u_\lambda\otimes u_\lambda$.  The coproduct positivity theorem of Lusztig~\cite{Lusztig}, Theorem~14.4.13(b), together with the canonical quotient maps, gives
\[
 \Gamma(b)=\sum_{\mu,\nu\in W\lambda}a_{\mu\nu}u_\mu\otimes u_\nu,
 \qquad a_{\mu\nu}\in\mathbb Z_{\geq0},
\]
for every specialized canonical vector $b$.  Composing $\Gamma$ with multiplication into $\Sym^2 V^\lambda$ gives the embedding in the statement.  Multiplication adds the coefficients of tensors with exchanged factors, so the coefficients remain nonnegative.

If $\dot w b$ is a scalar multiple of a canonical vector, its nonzero coefficients all have the same sign.  The permutation of unordered pairs $\{\mu,\nu\}\mapsto\{w\mu,w\nu\}$ cannot cancel distinct monomials, so all nonzero terms of $b$ acquire the same sign.  Thus $\Sym^2(D_\epsilon)b=\pm b$.  This holds for every canonical basis vector, proving the assertion about the Cartan square.
\end{proof}

\begin{proof}[Proof of Theorem~\ref{thm:canonical-obstruction}]
It suffices to use the Tits representative.  Another representative differs by a torus element, which acts by a nonzero scalar on each weight space.  It therefore permutes the canonical basis up to scalars if and only if the Tits representative does.  The action matrix on the dual basis is the inverse transpose.  An invertible matrix has exactly one nonzero entry in each row and column if and only if its inverse transpose does, so the canonical and dual canonical assertions are equivalent.

\emph{The vector representation of $D_n$.}
For $n\geq4$, use Bourbaki's numbering for type~$D_n$ and the coordinates of Example~\ref{ex:D4}, now with $1\leq i\leq n$.  The canonical zero-weight lines of $V^{2\omega_1}$ are
\[
 \C(z_1-z_2),\ \C(z_2-z_3),\ \ldots,\ \C(z_{n-1}-z_n).
\]
Indeed, for $1\leq i<n$, the weight-$\alpha_i$ space is one-dimensional, and $2\alpha_i$ does not occur.  Its canonical vector therefore lies at the top of its $i$-string, and lowering gives the line $\C(z_i-z_{i+1})$.  Canonical string compatibility gives these $n-1$ independent lines, the full dimension of $V^{2\omega_1}_0$.

The underlying coordinate permutation $\sigma$ of a type-$D_n$ Coxeter element has cycle type $(n-1)(1)$.  This is immediate for the standard product and follows for every Coxeter element by conjugacy.  The representative sends $z_i=x_i y_i$ to $z_{\sigma(i)}$, since the scalars on the two paired coordinates multiply to one.  Its action on the $z_i$ therefore has order $n-1\geq3$.  Preserving the displayed canonical lines would require an automorphism of the path $1--2--\cdots--n$, whose automorphism group has order two.  This proves the vector case.

\emph{The half-spin representations of $D_n$.}
It suffices to treat $\lambda=\omega_n$: the diagram automorphism interchanging the fork vertices preserves the canonical basis and interchanges the two half-spin representations.  Suppose that the Tits representative of a Coxeter element permutes the canonical basis up to scalars.  By Lemma~\ref{lem:canonical-sign}, the symmetric square of its diagonal sign map preserves the half-spin Cartan square.  The dual map therefore preserves the quadratic equations of the pure-spinor orbit.

The four-index Pfaffian equations impose a restriction on the underlying coordinate permutation $\sigma$.  As shown in Lemma~\ref{lem:spin-sign-obstruction}, preserving these equations forces $\sigma$ to be a cyclic rotation of $[n]$ or a cyclic rotation followed by reversal.  The calculation uses the inversion signs of $\sigma$ and is given in Appendix~\ref{app:canonical-certificates}.  But $\sigma$ has cycle type $(n-1)(1)$: it is neither a nonidentity rotation, which has no fixed point, nor a reflection, which has order at most two.  Since $n\geq4$, this is a contradiction.

\emph{The exceptional minuscule representations.}
For $E_6$ with $\lambda=\omega_1$ and $E_7$ with $\lambda=\omega_7$, Lemma~\ref{lem:exceptional-sign-certificate} gives four canonical vectors whose expansions in the products $u_\mu u_\nu$ each have exactly two terms, with equal positive coefficients.  For every orientation of the Dynkin diagram, its Tits representative gives opposite signs to the two terms of at least one of these vectors.  By Lemma~\ref{lem:canonical-sign}, no Coxeter element can therefore permute the canonical basis up to scalars.  The vectors, their top-of-string checks, and the complete sign recurrence appear in Appendix~\ref{app:canonical-certificates}; the sign calculations cover all orientations in each type.  Finally, the Dynkin involution interchanging $\omega_1$ and $\omega_6$ gives the second $E_6$ case.
\end{proof}

\section{Higher Cartan powers}\label{sec:higher-powers}

For higher Cartan powers $V^{m\lambda}$ with $m>2$, the remaining problem is to construct bases compatible with Coxeter-motion for every minuscule weight $\lambda$ and every Coxeter element.  The unpublished manuscript of GPT~5.6 Sol~\cite{Sol} proves the corresponding cyclic sieving identities by a case-by-case treatment of the minuscule families, with exact computer-assisted certificates in the exceptional types.  A uniform explanation through compatible bases remains open for general $m$.  

Given $\pi\in\RPP_m(P_\lambda)$, set $I_s:=\{p\in P_\lambda:\pi(p)<s\}$ and $\mu_s:=\wt(I_s)$ for $1\leq s\leq m$.  Then $\mu_1\geq\cdots\geq\mu_m$ is a multichain of weights of $V^\lambda$.  Can quiver data attached to these multichains produce bases of $V^{m\lambda}$ whose vectors transform with source and sink reflections and whose Coxeter action can be normalized to lift Coxeter-motion?

\appendix

\section{Classical cyclic sieving and RPP conventions}\label{app:classical-conventions}

We record the sign and grading calculations that identify the classical results in Section~\ref{subsec:csp-history} with the conventions of this paper.

\subsection{The ordinary long-cycle matrix and signed traces}

Return to the type~$A$ setting of Section~\ref{subsec:csp-history}, with $1\leq k\leq n-1$ and $m\geq1$, and use its notation $S_m$, $j$, and $\zeta_n$.  Let $g\in\operatorname{GL}_n(\C)$ be the ordinary permutation matrix of $c=(1\,2\,\cdots\,n)$.  On the dual canonical basis $\{F_T:T\in S_m\}$ of $S^{(m^k)}\C^n$, Rhoades~\cite{Rhoades}, Proposition~5.5, gives
\[
 gF_T=(-1)^{(k-1)T_n}F_{jT},
\]
where $T_i$ counts the entries of $T$ equal to $i$.  The author~\cite{RushGlobal}, Theorem~5.5, reproves this formula for the upper global basis.  By the restriction results in Sections~4.2--4.3 of that paper, the same argument applies to the lower global basis, and hence to the canonical basis at $q=1$.  For $m=1$, it is the ordinary wedge-reordering sign.

Suppose $0\leq d<n$ and $j^dT=T$.  Promotion cyclically rotates the content, so the content is constant on the cycles of addition by $d$ modulo $n$.  Writing $t:=\gcd(n,d)$, these are the residue classes modulo $t$.  Each occurs $d/t$ times among the last $d$ entries of $[n]$, and $T$ has $mk$ entries in total.  Thus
\[
 T_{n-d+1}+\cdots+T_n=\frac{dmk}{n}.
\]
The sign accumulated in $d$ applications of $g$ is consequently
\[
 (-1)^{(k-1)dmk/n}=\zeta_n^{dm\binom{k}{2}},
\]
independent of the fixed tableau.  Only fixed tableaux contribute to the trace, and $g$ is conjugate to $\operatorname{diag}(1,\zeta_n,\ldots,\zeta_n^{n-1})$.  It follows that
\begin{align*}
 \#\Fix(j^d)
 &=\zeta_n^{-dm\binom{k}{2}}\tr\!\left(g^d\mid S^{(m^k)}\C^n\right)\\
 &=\zeta_n^{-dm\binom{k}{2}}
 s_{(m^k)}(1,\zeta_n^d,\ldots,\zeta_n^{(n-1)d}).
\end{align*}
Periodicity gives the identity for all integers $d$.

\subsection{Order ideals and height-one RPPs}

For any finite poset $P$, the map $I\mapsto\pi_I:=\mathbf1_{P\setminus I}$ is a bijection from $J(P)$ to $\RPP_1(P)$.  It intertwines each order-ideal toggle with the corresponding piecewise-linear toggle, and $|\pi_I|=|P|-|I|$.  This is the convention used by Hopkins~\cite{Hopkins}, Section~4.2.

Now suppose that $P$ is self-dual, as every minuscule poset is, and choose an order-reversing bijection $\iota:P\to P$.  Then $\pi\mapsto m-\pi\circ\iota$ is a bijection of $\RPP_m(P)$ that replaces size $|\pi|$ by $m|P|-|\pi|$.  Hence
\begin{equation}\label{eq:rpp-reciprocity}
 F_{P,m}(q)=q^{m|P|}F_{P,m}(q^{-1}).
\end{equation}
In particular,
\begin{equation}\label{eq:height-one-grading}
 F_{P,1}(q)=\sum_{I\in J(P)}q^{|P|-|I|}=J(P;q).
\end{equation}
For minuscule $P$, the bijection therefore transfers the Rush--Shi theorem, with its generating polynomial, to height-one RPPs.

\subsection{Rectangular tableaux and threshold ideals}

Let $P=[k]\times[n-k]$, viewed as a rectangle with row lengths weakly decreasing from top to bottom.  For a $k$-subset $A=\{a_1<\cdots<a_k\}$ of $[n]$, let $I_A$ be the ideal whose row lengths are
\[
 (a_k-k,\ldots,a_1-1).
\]
This is a bijection from $k$-subsets to rectangular ideals, with $|I_A|=\sum_r(a_r-r)$.  Inclusion of ideals corresponds to componentwise comparison of the increasing subsets.

Given $\pi\in\RPP_m(P)$, form its threshold ideals $I_s:=\{p:\pi(p)<s\}$ for $1\leq s\leq m$.  Write $I_s=I_{A_s}$ and place $A_s$ in the $s$th column of a tableau $T$.  Each column is strictly increasing, and $I_1\subseteq\cdots\subseteq I_m$ makes the rows weakly increasing.  This gives a bijection with tableaux of shape $(m^k)$ and entries in $[n]$.

Label the cell in row $r$ and column $t$ of $P$ by $k-r+t$.  In these coordinates the correspondence can also be written
\[
 \pi(r,t)=\#\{s:T_{k+1-r,s}\leq k-r+t\}.
\]
Thus the RPP entries are cumulative row counts of tableau entries.  Label toggles are the piecewise-linear toggles on these counts, and their product $\tau_1\cdots\tau_{n-1}$, with the rightmost factor acting first, is promotion in our convention.  This is the Gelfand--Tsetlin description of the tableau correspondence in Hopkins~\cite{HopkinsPromotion}, Appendix~A, Propositions~A.7 and~A.9.  Consequently, $\gamma_{(1\,2\,\cdots\,n)}$ corresponds to $j$.

Finally, each $p\in P$ belongs to exactly $m-\pi(p)$ threshold ideals.  Therefore
\[
 \sum_{r=1}^k\sum_{s=1}^m(T_{rs}-r)
 =\sum_{s=1}^m|I_s|
 =mk(n-k)-|\pi|.
\]
The statistic on the left is the exponent in the shifted Schur polynomial~\eqref{eq:rhoades-polynomial}.  Equation~\eqref{eq:rpp-reciprocity} therefore gives
\[
 q^{-m\binom{k}{2}}s_{(m^k)}(1,q,\ldots,q^{n-1})
 =F_{[k]\times[n-k],m}(q),
\]
as asserted in Section~\ref{subsec:csp-history}.

\section{Sign calculations for the canonical obstruction}
\label{app:canonical-certificates}

This appendix supplies the two sign calculations used in the proof of Theorem~\ref{thm:canonical-obstruction}.  The half-spin calculation reduces preservation of the Cartan square to a condition on the inversions of a permutation.  In the exceptional types, four explicit canonical vectors suffice to exclude every orientation.

\subsection{The half-spin sign obstruction}

\begin{lemma}\label{lem:spin-sign-obstruction}
Let $n\geq4$, and realize type~$D_n$ in $\R^n$ with standard orthonormal basis $\varepsilon_1,\ldots,\varepsilon_n$ and simple roots $\alpha_i=\varepsilon_i-\varepsilon_{i+1}$ for $i<n$ and $\alpha_n=\varepsilon_{n-1}+\varepsilon_n$.  Let $\lambda=\omega_n$, and let $c$ be a Coxeter element.  Write
\[
 \dot c\,u_\mu=\epsilon(\mu)u_{c\mu}
 \qquad(\mu\in W\lambda),
\]
and let $\sigma\in\mathfrak S_n$ be the image of $c$ under the homomorphism that forgets the signs of a signed coordinate permutation.  Set $D_\epsilon u_\mu:=\epsilon(\mu)u_\mu$.  If $\Sym^2(D_\epsilon)$ preserves the Cartan square, then $\sigma$ belongs to the permutation group generated by a cyclic rotation of $[n]$ and a reversal.
\end{lemma}

\begin{proof}
Index the half-spin weights by even subsets $A\subseteq[n]$.  Writing $a_i:=\mathbf1_{i\in A}$, set
\[
 \mu_A:=\frac12\sum_{i=1}^n(1-2a_i)\varepsilon_i.
\]
The corresponding standard weight vector is $u_{\mu_A}$.  For $i<n$, the simple lowering operator $f_i$ changes $(a_i,a_{i+1})=(0,1)$ to $(1,0)$; the operator $f_n$ changes $(a_{n-1},a_n)=(0,0)$ to $(1,1)$.  All these nonzero coefficients are $1$.  Thus, for the convention $n_i=\exp(e_i)\exp(-f_i)\exp(e_i)$, the signs $\epsilon_i(A):=\eps_i(\mu_A)$ of the simple Tits actions are
\begin{equation}\label{eq:spin-simple-signs}
 \epsilon_i(A)=
 \begin{cases}
 (-1)^{(1-a_i)a_{i+1}},&i<n,\\
 (-1)^{(1-a_{n-1})(1-a_n)},&i=n.
 \end{cases}
\end{equation}
On the tuple $(a_1,\ldots,a_n)$, the simple reflection swaps entries $i,i+1$ for $i<n$; for $i=n$ it swaps entries $n-1,n$ and replaces both by their complements.  Equations~\eqref{eq:spin-simple-signs} give
\begin{equation}\label{eq:spin-coxeter-sign}
 \dot c\,u_{\mu_A}=\epsilon(A)u_{c\mu_A},\qquad
 \epsilon(A)=(-1)^{\sum_{r<s}e_{rs}a_ra_s+\ell(A)},\qquad
 e_{rs}:=\mathbf1_{\sigma(r)>\sigma(s)},
\end{equation}
where $\epsilon(A):=\epsilon(\mu_A)$ and $\ell$ is affine-linear over $\mathbb F_2$.

To verify~\eqref{eq:spin-coxeter-sign}, observe that the quadratic part of each exponent in~\eqref{eq:spin-simple-signs} is $a_i a_{i+1}$, with positions $n-1,n$ used for the fork generator.  This is the sign exponent for the same adjacent transposition acting on an exterior monomial.  These exterior signs compose to the inversion exponent of $\sigma$.  Replacing any $a_i$ by $1-a_i$ in a quadratic polynomial changes it only by affine-linear terms, so the additional complements do not change this quadratic part.

By hypothesis, $\Sym^2(D_\epsilon)$ preserves the Cartan square.  Its dual therefore preserves the quadratic equations of the pure-spinor orbit in $\mathbb P(V^{\omega_n})$: the Cartan square is the span of the squares of vectors in the highest-weight orbit, so its annihilator consists precisely of these equations.

Write $p_A$ for the coordinate dual to $u_{\mu_A}$.  For any four indices $i<j<k<l$, consider the four monomials
\[
 p_\varnothing p_{ijkl},\qquad
 p_{ij}p_{kl},\qquad p_{ik}p_{jl},\qquad p_{il}p_{jk}.
\]
Their relation space on the pure-spinor orbit is one-dimensional, spanned by
\begin{equation}\label{eq:spin-four-index}
 p_\varnothing p_{ijkl}-p_{ij}p_{kl}
 +p_{ik}p_{jl}-p_{il}p_{jk}=0.
\end{equation}
Indeed, in the exterior-algebra realization of the half-spin module, identify $u_{\mu_A}$ with the ordered wedge of the vectors $\mathbf e_a$, $a\in A$, where $\mathbf e_1,\ldots,\mathbf e_n$ is the standard basis of $\C^n$ and $u_{\mu_\varnothing}=1$.  This agrees with the lowering coefficients above.  On the standard affine chart $\exp(\sum_{r<s}x_{rs}\mathbf e_r\wedge\mathbf e_s)$, one has
\[
 p_\varnothing=1,\qquad p_{ij}=x_{ij},\qquad
 p_{ijkl}=x_{ij}x_{kl}-x_{ik}x_{jl}+x_{il}x_{jk}.
\]
Since the $x_{rs}$ are independent coordinates, there is exactly the displayed relation.

Since $\Sym^2(D_\epsilon)$ is diagonal, its dual preserves the span of these four monomials.  It must therefore take their unique relation to a scalar multiple of itself, so the four products of signs agree.  The affine-linear terms $\ell$ cancel when any two products are compared, since each of the four indices occurs once in each product.  Consequently, in $\mathbb F_2$,
\begin{equation}\label{eq:spin-four-point}
 e_{ij}+e_{kl}=e_{ik}+e_{jl}=e_{il}+e_{jk}
 \qquad(i<j<k<l).
\end{equation}
To solve these equations, extend $e_{ij}$ symmetrically to unordered pairs.  Applying~\eqref{eq:spin-four-point} to $\{1,i,j,k\}$ shows that all sums $e_{1i}+e_{1j}+e_{ij}$, for distinct $i,j>1$, have the same value $v$.  Set $u_1:=0$ and $u_i:=e_{1i}+v$ for $i>1$.  Then
\begin{equation}\label{eq:spin-cut}
 e_{ij}=v+u_i+u_j\qquad(i\ne j).
\end{equation}

If $v=0$, the two classes defined by $u_i$ cannot interleave.  Otherwise there are $i<j<k$ with $u_i=u_k\ne u_j$, and $e_{ij}=e_{jk}=1$ but $e_{ik}=0$, contradicting $\sigma(i)>\sigma(j)>\sigma(k)$.  Thus the classes are consecutive blocks; $\sigma$ is increasing on each block and reverses their order, so it is a cyclic rotation of $[n]$.  If $v=1$, composing $\sigma$ with the full reversal replaces each $e_{ij}$ by $1-e_{ij}$ and reduces to the preceding case.  Hence~\eqref{eq:spin-four-point} implies that $\sigma$ belongs to the permutation group generated by a cyclic rotation and a reversal.
\end{proof}

\subsection{The exceptional sign certificates}

\begin{lemma}\label{lem:exceptional-sign-certificate}
Take $\lambda=\omega_1$ in type $E_6$ or $\lambda=\omega_7$ in type $E_7$.  In each type there are four canonical vectors, listed below, whose expansions in the products $u_\mu u_\nu$ each have exactly two terms, with equal positive coefficients.  For every orientation $Q$, the image of at least one of these vectors under $\dot c_Q$ has one positive and one negative coefficient in these products.
\end{lemma}

\begin{proof}
Use Bourbaki numbering, with edges
\[
 (1,3),\ (2,4),\ (3,4),\ (4,5),\ (5,6)
 \quad\text{and, in }E_7,\quad (6,7).
\]
Weights below are given by their simple-coroot coordinates.  For $\mu\in W\lambda$ and a word $\mathbf i=(i_1,\ldots,i_s)$, let
\[
 G(\mu;\mathbf i):=f_{i_s}\cdots f_{i_1}u_\mu^2,
\]
considered up to a positive scalar.  The word records the operators in the order in which they act.

The eight vectors in the table are canonical basis vectors up to a positive scalar.  To check this, start with the extremal canonical vector of weight $2\mu$ in the quantum module, embedded in the tensor square of the minuscule module.  At each successive lowering step $F_i$, every factor in every tensor term has $i$-coordinate $0$ or $1$, so $E_i$ annihilates the vector term by term.  Since $F_i$ acts at the top of an $i$-string, it preserves the lower-global basis.  The quantum coproduct contributes powers of $q$ with positive coefficients, so the classical calculations below give the same supports.  Specializing at $q=1$ and multiplying the tensor factors gives $G(\mu;\mathbf i)$.

The top-of-string checks have only three forms.  The two-letter words use orthogonal nodes with initial coordinates $(1,1)$.  The word $(3,1,6,5)$ uses the two orthogonal $A_2$ chains $3$--$1$ and $6$--$5$, with initial coordinates $(1,0)$ on each chain.  Finally, $(4,3,5,4)$ uses the $A_3$ chain $3$--$4$--$5$, with initial coordinates $(0,1,0)$.  These checks also show that each resulting vector has exactly two ambient monomials with equal positive coefficients.  More explicitly, every row has the cube support
\[
 u_\mu u_{\mu-\beta-\gamma}+u_{\mu-\beta}u_{\mu-\gamma},
\]
where $(\beta,\gamma):=(\alpha_i,\alpha_j)$ for a two-letter word $(i,j)$, $(\alpha_1+\alpha_3,\alpha_5+\alpha_6)$ for $(3,1,6,5)$, and $(\alpha_4,\alpha_3+\alpha_4+\alpha_5)$ for $(4,3,5,4)$.  In each case the two roots are orthogonal.  For example, in the last case the first three operators give $f_5f_3f_4(u_\mu^2)=2u_\mu u_{\mu-\gamma}$.  Both factors have $4$-coordinate $1$, so the last operator gives
\[
 f_4f_5f_3f_4(u_\mu^2)
 =2\bigl(u_{\mu-\beta}u_{\mu-\gamma}
          +u_\mu u_{\mu-\beta-\gamma}\bigr).
\]

Set $r=6$ or $7$ in the respective cases.  Index orientations by binary words $b_1\cdots b_{r-1}$ using the displayed order of the edges: $0$ directs the smaller vertex toward the larger, and $1$ reverses that direction.  Set $t:=(b_1\cdots b_{r-1})_2$.  For each vector, let $S$ be the set of indices $t$ for which $\dot c_Q$ gives opposite signs to its two monomials.  The last column records the integer $\sum_{t\in S}2^t$ in hexadecimal.  It thus specifies the subset of the $32$ or $64$ orientations excluded by that vector.

\begin{center}
\small
\setlength{\tabcolsep}{5pt}
\begin{tabular}{ccll}
Type & $\mu$ & $\mathbf i$ & Excluded orientations\\\hline
$E_6$ & $(0,-1,0,1,-1,1)$ & $(4,6)$ & \texttt{64866446}\\
 & $(1,0,-1,1,-1,0)$ & $(1,4)$ & \texttt{0f0e30c1}\\
 & $(0,0,1,-1,0,1)$ & $(3,1,6,5)$ & \texttt{80308338}\\
 & $(0,-1,0,1,0,-1)$ & $(4,3,5,4)$ & \texttt{1c412c00}\\[3pt]
$E_7$ & $(0,1,1,-1,0,0,0)$ & $(2,3)$ & \texttt{f20d0df20d00f200}\\
 & $(1,1,-1,0,0,0,0)$ & $(1,2)$ & \texttt{000d00f2ff0dfff2}\\
 & $(0,1,0,-1,1,-1,1)$ & $(2,7)$ & \texttt{0e0ececece0e0ece}\\
 & $(-1,-1,0,1,0,0,0)$ & $(4,3,5,4)$ & \texttt{03f030030cf00001}\\\hline
\end{tabular}
\end{center}

The following recurrence checks the subsets recorded in the table.  Choose a source order $i_1,\ldots,i_r$ for the orientation and set $c:=c_Q=s_{i_r}\cdots s_{i_1}$.  For $\nu\in W\lambda$, compute
\[
 \nu^{(0)}:=\nu,\qquad
 \nu^{(j)}:=s_{i_j}\nu^{(j-1)}\quad(1\leq j\leq r),\qquad
 \epsilon_c(\nu):=(-1)^{\#\{j:(\nu^{(j-1)})_{i_j}=1\}}.
\]
In simple-coroot coordinates the reflection is $(s_i\nu)_k=\nu_k-\nu_i a_{ik}$, where $(a_{ik})$ is the Cartan matrix.  Thus $\dot c u_\nu=\epsilon_c(\nu)u_{c\nu}$ by~\eqref{eq:standard-tits}.  If the vector has support $u_{\nu_1}u_{\nu_2}+u_{\nu_3}u_{\nu_4}$, the orientation index $t$ belongs to $S$ precisely when
\[
 \epsilon_c(\nu_1)\epsilon_c(\nu_2)\epsilon_c(\nu_3)\epsilon_c(\nu_4)=-1.
\]
This recurrence uses only the given edge list and vectors.  The result is independent of the chosen source order, since any two source orders differ by interchanging consecutive vertices that are not joined by an edge, whose Tits representatives commute.  The union of the four subsets is encoded by $\texttt{ffffffff}=2^{32}-1$ in type $E_6$ and $\texttt{ffffffffffffffff}=2^{64}-1$ in type $E_7$.  Thus every orientation sends at least one canonical vector to a vector whose two nonzero coefficients in the products $u_\mu u_\nu$ have opposite signs.
\end{proof}

\section{Restriction to types \texorpdfstring{$B$ and $C$}{B and C}}\label{sec:folding}

For completeness, we record how the quiver bases restrict to the minuscule representations of types $B$ and $C$.  These representations introduce no new minuscule posets, and the restriction argument is not needed for the uniform result of Section~\ref{sec:csp}.  The relevant inclusions of fixed-point subgroups are:
\[
 B_n\subset D_{n+1}\quad\text{(spin inside half-spin)},
 \qquad
 C_n\subset A_{2n-1}\quad\text{(the vector representation)}.  
\]
Here $n\geq2$, with $D_3$ interpreted as $A_3$; the rank-one case is already of type~$A_1$.  The inclusions refer to the corresponding Lie algebras and simply connected groups.  In either inclusion, write $\widetilde\lambda$ for the highest weight of the indicated ambient representation and $\lambda$ for its restriction to the fixed Cartan subalgebra.  The ambient minuscule representation restricts irreducibly, and its underlying minuscule poset is unchanged; write $P$ for this common poset.  The same holds for the Cartan square $V^{2\widetilde\lambda}$: its highest vector generates a submodule of highest weight $2\lambda$, and the dimension formula~\eqref{eq:dimension}, applied to the common poset, gives equality of dimensions.  For the type~$B$ restriction, these identifications are also recorded by Lax~\cite{Lax}, Proposition~21.

The diagram involution combines the two fork vertices of $D_{n+1}$ into the short-root label of $B_n$, and combines vertices $i$ and $2n-i$ of $A_{2n-1}$ into a label of $C_n$.  Distinct vertices in each orbit are nonadjacent.  With compatible Chevalley generators, the generators of the fixed-point Lie algebra are the sums of the corresponding generators in each orbit.  Generators attached to distinct vertices of an orbit commute, so the folded Tits representative is the product of the Tits representatives in that orbit.  The folded label toggle is likewise the product of the commuting label toggles in its orbit.  

Lift an orientation of the folded diagram to the simply laced diagram.  Expanding each folded label into its orbit converts a folded source sweep into a simply laced source sweep, identifying both their Tits representatives and their Coxeter-motions.  The Coxeter numbers on both sides of each inclusion are $2n$.  Writing $\widetilde\rho^\vee$ and $\rho^\vee$ for the respective half-sums of positive coroots, the identity
\[
 \langle2\widetilde\lambda,\widetilde\rho^\vee\rangle
 =|P|=\langle2\lambda,\rho^\vee\rangle
\]
makes the scalar normalization agree on the two sides.  The simply laced quiver basis, regarded in the restricted Cartan square, therefore satisfies~\eqref{eq:intro-coxeter}.  Thus the normalized lifting formula also holds for these restricted representations.  To construct these basis vectors, one may lift the source order, apply reflect-and-peel in the simply laced diagram, and restrict the resulting cube vectors.

\section*{Acknowledgments}

The author is grateful to GPT-6 Astra for the productive mathematical collaboration that led to this article.  The abstract and introduction were written entirely by the author; the body was written jointly.  The author assumes full responsibility for all content.

The author also thanks Victor Reiner for introducing him, some fifteen years ago, to the work of Robert Proctor and John Stembridge on minuscule posets.

\end{document}